\documentclass{article}

\usepackage{latexsym}
\usepackage{amsmath}
\usepackage{amsfonts}
\usepackage{graphicx}
\usepackage{color}
\usepackage{hyperref}
\usepackage{enumitem}
\newlist{enumthm}{enumerate}{1}
\setlist[enumthm]{label=(\alph{*})}
\usepackage{theorem}

\renewcommand{\div}{\operatorname{div}}

\theoremheaderfont{\scshape}
\newtheorem{theorem}{Theorem}[section]

\newtheorem{lemma}[theorem]{Lemma}

\newtheorem{proposition}{Proposition}
\newtheorem{algo}[theorem]{Algorithm}
\theorembodyfont{\normalfont}
\newtheorem{remark}{Remark}[section]
\newtheorem{exemple}{Exemple}[section]

\newenvironment{proof}{{\bf Proof }}{\hbox{~} \hfill \rule{0.5em}{0.5em}\\}

\begin{document}

\title{Existence and uniqueness of solution for semi  linear conservation laws with velocity field in $L^{\infty} $ 
\footnote{{This work is supported by  NLAGA Project (Non Linear Analysis, Geometry and Applications Project).\newline
The authors declare that there is no conflict of interest regarding the publication of this paper.
}}}         

\date{}          

\maketitle

\centerline{\scshape Souleye Kane \footnote{souleye@ucad.edu.sn}}
\medskip
{\footnotesize
\centerline{ Universit\'e Cheikh Anta Diop de Dakar.}
\centerline{  Faculty des Sciences et Techniques }
} 
\medskip
\centerline{\scshape S.Fallou Samb  \footnote{serignefallou.samb@ucad.edu.sn}}
{\footnotesize
\centerline{ Universit\'e Cheikh Anta Diop de Dakar.}
\centerline{Departement de Math\'ematiques et Informatique}
 } 
 
\medskip
\centerline{\scshape Diaraf  Seck \footnote{diaraf.seck@ucad.edu.sn}}
\medskip
{\footnotesize
 \centerline{Universit\'e Cheikh Anta Diop de Dakar, BP 16889 Dakar Fann,}
  \centerline{  Ecole Doctorale de
                   Math\'ematiques et Informatique. }

}
\pagestyle{myheadings}
 \renewcommand{\sectionmark}[1]{\markboth{#1}{}}
\renewcommand{\sectionmark}[1]{\markright{\thesection\ #1}}
\begin{abstract}\noindent In this paper we extend results obtained in \cite{BP} and \cite{KB}. By considering a semi linear conservation law with velocity in $ L^{\infty} $, we prove by fixed point arguments existence and uniqueness result and even in a penalized situation.
\end{abstract}
Keywords: transport equations, semi linear PDE,  fixed point methods, STILS method, conservation laws, advection-reaction, finite element method, Newton's method, Picard's iteration.\\
\section{Introduction }
This paper deals  about semi linear conservations laws  with velocity field in $L^{\infty}.$ Our goal is  twofold. On the one hand, the focus is to propose a  generalization of space time integrated  least square(STILS)  method introduced by O. Besson and J. Pousin in \cite{BP}  for  linear conservation laws to semi linear ones. The STILS method has been widely studied    in numerous linear cases. Our aim is to introduce a non linearity  in the source term and look for  theoretical methods   to prove existence and uniqueness results. For this, we shall propose  methods combining  variational and topological methods. \\
 To reach this aim, we shall use two fixed point  theorems. The first one is the Banach's fixed point theorem and the second is due to  Schauder. In this latter case, we shall need a penalization argument. \\
 On the other hand, we endeavor  to propose  numerical methods to analyse  semi linear boundary value problems. We shall use finite element methods combined with Picard's iteration and Newton's methods.\\ 
 Finite element method is known to produce spurious oscillations and add diffusions in the orthogonal directions of integral curve when convection-dominated problem is solved see \cite{LP} and  references therein. To remedy it, the space time integrated least square method  has been introduced in  finite element context  by  H. Nguyen and J. Reynen in \cite{HJ} for solving advection-diffusion equation.
And a time marching approach of STILS   has been proposed by  O. Besson and  G. De  Montmollin  in \cite{GM} for solving numerically  linear transport equation using the finite element method with $div(u)=0$. To  get discrete maximum principle and  remove the oscillations produced by the STILS method, J. Pousin, K. Benmansour, E. Bretin and L. Piffet in \cite{KB}, added to the formulation a constraint of positivity and a penalization of the total variation.\\
Before presenting the organization of our work, let us point out that interesting works  on the SILS method have been already  realized. We quote some among them closely related to our theoretical  works. In fact,  it  been has been  used by P.Azerad and O. Besson in \cite{AZ1} to give a coercive variational formulation to the transport equation with a free divergence  $\mathcal C ^1$ regular   velocity vector field.  
Existence and uniqueness of Space Time Least Square solution of linear conservation law with velocity field in $L^{\infty}$  is proved   in \cite{BP} by O. Besson and J. Pousin. And in the same paper, these latter deduce  a   maximum  principle result  from Stampacchia’s theorem and   have established
the comparison between the least squares solution  and the renormalized solution of these equations.\\
 

The paper is organized as follows. In the next section we shall do the presentation of the problem with some useful mathematical tools for our study. The third section is devoted to the existence and uniqueness results. The main used arguments are fixed point theorems (Banach-Picard's  theorem, Schauder's Theorem). And in  the last section, we propose  two  news numerical methods for computing the solution by using fixed point algorithm.    
\section{Position of the problem}
\subsection{Statement of the  aim and  functional setting}
Let  $\Omega \subset \mathbb{R}^{d}, (d \in \mathbb{N}^{*} )$  be a domain with a Lipschitz  boundary $ \partial \Omega $ satisfying the cone proprety.  Let us take $T>0$ , a set  $  Q= \Omega \times ]0,T[ $ and  consider an advection velocity $ u:Q \rightarrow \mathbb{R}^{d} $ with the following regularity  property
$$u\in L^{\infty}(Q)^{d}\ \text{with}\  \div(u)\in L^{\infty}(Q).    $$
Let $ f:\mathbb{R}\longrightarrow \mathbb{R} $  be a function such that  $ f \in W^{1,\infty} (\mathbb{R}). $ In some  situations we can consider   $f$ as a $ k-$ Lipschitz,  for $ k $ small enough.\\

The first question we will look is to  find a space time least square solution for  the following boundary value problem   

\begin{equation}\label{SLCL}
\left \{ \begin{array}{l} 
 \frac{\partial c}{\partial t}+\div(uc)=f(c)  \text\ {in}\ Q \\
c(x,0)= c_{0}(x) \mbox{ in } \Omega\\
c(x,t)=c_{1}(x,t) \ \mbox{on} \ \Gamma_{-} \end{array}
\right .
\end{equation}
where
$$ \Gamma_{-} = \{x \in \partial \Omega :(n(x),u(x,t))< 0; \, \,   \forall t \in (0, T)  \}. $$ 

and $ (.,.) $  is the inner product in  $ \mathbb{R}^{d} $, $ n(x) $ is the outer normal to $ \partial \Omega $ at point $x $. For the sake of simplicity one assume that $  \Gamma_{-} $ does not depend on the time t.\\
Let us  consider $$u\in L^{\infty}(Q)^{d}\ \text{such that}\  \div(u)\in L^{\infty}(Q),$$   set $\widetilde{u}=(1,u_{1},u_{2},....,u_{d}) $ and $\widetilde{n}(x,t)$ the outer normal to $\partial Q$ at $(x,t)$.\\

We shall use the notation $\mid E \mid$ to mean the Lebesgue measure of a set $E$ throughout this paper.
Let us recall that,  the space-time incoming flow boundary is given by   $$\partial Q_{-}=
\big\lbrace (x,t)\in \partial Q ,(\widetilde{u}(x,t),\widetilde{n}(x,t)  ) <0  \big\rbrace    = \Omega \times \lbrace 0 \rbrace \cup \Gamma_{-} \times ]0,T[.  $$  
The incoming flow boundary condition in space-time is defined as follows  
$$c_{b}(x,t)= \left \{ \begin{array}{l}
c_{0}(x) \mbox{ if } t=0\\ c_{1}(x,t) \ \mbox{on} \ \Gamma_{-} \end{array}
\right .$$

 We introduce the following norm   defined by: 
 \begin{enumerate}
 \item 
$  \Vert \phi \Vert^{2} =\Vert \phi \Vert_{L^{2}(Q)} ^{2}+\Vert\widetilde{div}(\widetilde{u}\phi)  \Vert_{L^{2}(Q)} ^{2} -\displaystyle \int_{ \partial Q_{-} } \phi^{2}(\widetilde{u},\widetilde{n}) \mathrm{d}s   $
for all $ \phi \in  D(\overline{Q}) .$\\ Where 
\item $\,\,\,\,\,\, \widetilde{\nabla}\phi=(\frac{\partial \phi}{\partial t},\frac{\partial \phi}{\partial x_{1}} ,\frac{\partial \phi}{\partial x_{1}},..    ,\frac{\partial \phi}{\partial x_{d}}  ); $ 
\item $\,\,\,\,\,\, \widetilde{\div}(\widetilde{u}\phi)= \frac{\partial \phi}{\partial t}+\displaystyle  \sum_{i=1}^{d}\frac{\partial (\phi u )}{\partial x_{i}} ;$

\item  And the Sobolev space $ H(u,Q) =\overline{D(\overline{Q})}^{  \Vert. \Vert };  $  

\item  Note	 that if $u$ is  regular enough for  instance $u\in L^{\infty}(Q)^{d}\ \text{with}\  \div(u)\in L^{\infty}(Q),    $ then  $\,\,\,\,\,\, H(u,Q)\cap L^{\infty}=\big\lbrace \phi \in L^{2}(Q);  \widetilde{\div}(\widetilde{u}\phi)\in L^{2}(Q),\phi_{/ \partial Q_{-}}  \in  L^{2}(\partial Q_{-};\mid(\widetilde{u},\widetilde{n})\mid  ) \rbrace \cap L^{\infty}  $
  for more details see \cite{BP}.

\end{enumerate}
Before proceeding further, let us  remind   the following theorems that will be useful for our work and for their proofs, we invite the reader to see \cite{BP}.
\begin{theorem}\label{trace21}
Let us consider  $u\in L^{\infty}(Q)^{d}\ \text{with}\  \div(u)\in L^{\infty}(Q) $.Then  the normal trace of u $ (\widetilde{u},\widetilde{n}) \in L^{\infty}( \partial Q).$
\end{theorem}



\begin{theorem}\label{trace}
Let  $u\in L^{\infty}(Q)^{d}\ \text{with}\  \div(u)\in L^{\infty}(Q) $. Then 
there exists  a linear continuous trace operator 

$$\begin{array}{clcl}
\gamma_{\widetilde{n}} : &H(u,Q)   &\longrightarrow  &L^{2}(\partial Q;(\widetilde{u},\widetilde{n})  )\\
& \phi   &\longmapsto   &\phi_{/ \partial Q}\\
\end{array}$$
 which can be localized as

$$\begin{array}{clcl}
\gamma_{\widetilde{n}\pm} : &H(u,Q)   &\longrightarrow  &L^{2}(\partial Q_{\pm};(\widetilde{u},\widetilde{n})  )\\
& \phi   &\longmapsto   &\phi_{/ \partial Q_{\pm}}.
\end{array}$$
\end{theorem}

Finally, let us  define the spaces 
$$ H_{0}(u,Q,\partial Q_{-})=\big\lbrace \phi \in     H(u,Q),\phi=0 \ \text{on } \ \partial Q_{-} \big\rbrace=H(u,Q)\cap Ker \gamma_{\widetilde{n}_{-}} $$
and 
$$G_{\pm}=\gamma_{\widetilde{n}_{\pm}}(H(u,Q)).   $$

Let us give the curved inequality still called curved Poincar\'e inequality,  below  that  is   fundamental and even is the precursor of existence of STILS solution. It has been introduced and proved  in \cite{AZ1} for free-divergence   and extended in \cite{BP}. \\
There exists $c_p >0$ such that for any  $\phi \in H(u,Q):$
\begin{equation}\label{constp}
\Vert \phi \Vert_{L^{2}(Q)} ^{2} \leq c_{p}^{2} \big ( \Vert\widetilde{\div}(\widetilde{u}\phi)  \Vert_{L^{2}(Q)} ^{2} -\int_{ \partial Q_{-} } \phi^{2}(\widetilde{u},\widetilde{n}) \mathrm{d}s \big ) 
\end{equation}
From  the curved inequality one deduces  the following theorem.

\begin{theorem}\label{poinco}
Let  $u\in L^{\infty}(Q)^{d}\ \text{with}\  \div(u)\in L^{\infty}(Q) $.Then 
the semi norm on $H(u,Q) $ defined by 
$$  \mid \phi \mid_{1,u}^{2} = \Vert \widetilde{\div}(\widetilde{u}\phi) \Vert_{L^{2}(Q)} ^{2} -\int_{ \partial Q_{-} } \phi^{2}(\widetilde{u},\widetilde{n}) \mathrm{d}s  $$
is a square of norm, equivalent to the norm defined  on $H(u,Q). $

\end{theorem} 

Thus $ H(u,Q)  $ can be  equipped by the norm  $ \mid . \mid_{1, u}. $
\begin{remark}\label{rq1}
In the free-divergence case, one gets  that $ c_{p}\leq 2T $ (see for  instance  \cite{AZ1} for additional information).
\end{remark}

\subsection{Space time least square and  the linear problem}
In this section, we are going to recall  the design and some proprieties of STILS  method for solving the following  linear conservation laws  problem.

\begin{equation}\label{tr}
\left \{ \begin{array}{l} 
 \frac{\partial c}{\partial t}+\div(uc)=f  \text\ {in}\ Q \\
c(x,0)= c_{0}(x) \mbox{ in } \Omega\\
c(x,t)=c_{1}(x,t) \ \mbox{on} \ \Gamma_{-}.\end{array}
\right .
\end{equation}

The space time least square solution of (\ref{tr}) corresponds to a minimizer in

$$\big\lbrace \phi \in H(u,Q); \gamma_{\widetilde{n}_{-}}(\phi)=c_{b}   \big \rbrace  $$
of the following convex, $H(u,Q)$-coercive functional  defined by 

\begin{equation}\label{eq:fonc}
J(c)=\dfrac{1}{2} \big ( \int_{Q}(\widetilde{\div}(\widetilde{u}c)-f)^{2}\mathrm{d}x\mathrm{d}t-\int_{ \partial Q_{-} } c^{2}(\widetilde{u},\widetilde{n}) \mathrm{d}s       \big)
\end{equation}

The Gâteaux-differential of J yields

\begin{equation}\label{eq:diff}
D[J(c)].\phi= \int_{Q}(\widetilde{\div}(\widetilde{u}c)-f)\widetilde{\div}(\widetilde{u}\phi)\mathrm{d}x\mathrm{d}t  -\int_{ \partial Q_{-} } c\phi(\widetilde{u},\widetilde{n}) \mathrm{d}s.  
\end{equation}

Thus, if $c_{b} \in G_{-} $, the space time  last square formulation of (\ref{tr}) is   expressed as follows

\begin{equation}\label{eq:stils}
\int_{Q}\widetilde{\div}(\widetilde{u}c)\widetilde{\div}(\widetilde{u}\phi)\mathrm{d}x\mathrm{d}t  =\int_{Q}f\widetilde{\div}(\widetilde{u}\phi)\mathrm{d}x\mathrm{d}t  \,\, \forall \,\,  \phi \in H_{0}(u,Q)  
\end{equation}

and 
$$ \gamma_{\widetilde{n}_{-}}(c)=c_{b}. $$ 
For more details  see \cite{BP}, \cite{KB}. 

Thanks to  Theorem \ref{trace} we can reduce the problem (\ref{eq:stils}) in a  homogeneous one in $\partial Q_{-}.$ For $c_{b} \in G_{-} $, let $C_{b} \in H(u,Q)$ such that  $\gamma_{\widetilde{n}_{-}}(C_{b}) =c_{b}$ then $ \rho=c-C_{b} $ is the unique solution of 
 
\begin{equation}\label{eq:stilsd}
\int_{Q}\widetilde{\div}(\widetilde{u}\rho)\widetilde{\div}(\widetilde{u}\phi)\mathrm{d}x\mathrm{d}t  =\int_{Q}(f-\widetilde{\div}(\widetilde{u}C_{b} ))\widetilde{\div}(\widetilde{u}\phi)\mathrm{d}x\mathrm{d}t  \,\, \forall \,\,  \phi \in H_{0}(u,Q).  
\end{equation}
Finally let us recall the following theorem proved in \cite{BP}.

\begin{theorem}\label{existence}
For  $u\in L^{\infty}(Q)^{d}\ \text{with}\  \div(u)\in L^{\infty}(Q) $, $ c_{b} \in G_{-}$, and $f\in L^{2}(Q)$, the problem (\ref{eq:stilsd})  has a unique solution. Moreover 
$$\mid \rho \mid_{1,u} \leq  \Vert f \Vert_{L^{2}(Q)}  + \Vert \widetilde{\div}(\widetilde{u}C_{b} ) \Vert_{L^{2}(Q)}  $$ 
and the function $c=\rho+C_{b}$ is the espace-time least squares solution of (\ref{tr}).
\end{theorem}

\subsection{Space time least square and semi linear problem }
This  last subsection is devoted to  introduce a variational formulation (\ref{SLCL}).  Otherwise 
our aim is to find $ c \in H(u,Q) $ such that 

\begin{equation}\label{eq:stilsn}
\int_{Q}\widetilde{\div}(\widetilde{u}c)\widetilde{\div}(\widetilde{u}\phi)\mathrm{d}x\mathrm{d}t  =
\int_{Q}f(c)\widetilde{\div}(\widetilde{u}\phi)\mathrm{d}x\mathrm{d}t  \,\, \forall \,\,  \phi \in H_{0}(u,Q,\partial Q_{-} )  
\end{equation}
 and

\begin{equation}\label{CD}
\gamma_{\widetilde{n}_{-}}(c) =c_{b}
\end{equation} 
It is important to stress that the above formulation is nonlinear. And we shall propose fixed point methods to study it. Let us recall that there are at least three distinct classes of such abstract theorems that are useful for proving existence results in  a wide  family  of partial differential equations.  These classes are
\begin{itemize}
\item fixed point theorems for strict contractions,
\item fixed point theorems for compact mappings and
\item fixed point theorems for order preserving operators.
\end{itemize}
We shall use in the following  the two first types.

\section{Existence and   qualitative results}
\subsection{Existence and uniqueness}
In this section we shall study the problem  (\ref{eq:stilsn}) by establishing and proving  existence and uniqueness theorems for the STILS solution. Theses results are deduced  thanks to  the fixed point theory namely the Banach-Picard and Schauder theorems.\\

At first,  in the case where $f$ is $k-$ Lipschitz with $ k $ is  small enough that will be precised and  by using the Banach-Picard fixed point theorem (\cite{Evans} ), we have the following existence and uniqueness  theorem  of STILS solution. 

\begin{theorem}
Let  $ u\in L^{\infty}(Q) $ with $ \div(u)\in L^{\infty}(Q) $ , and $ c_{b} \in G_{-} $ , $f$ be $k-$ Lipschitz in $ \mathbb{R} $ with $ k<\frac{1}{c_{p}} $. Then the problem (\ref{eq:stilsn})-(\ref{CD}) has a unique solution.  
\end{theorem}
\begin{proof}

Let us consider $$\mathbb{H}=\big\lbrace \phi \in   H(u,Q), \gamma_{\widetilde{n}_{-}}(\phi) =c_{b}\big\rbrace   $$
For all $\rho \in \mathbb{H}  $ ,$ f(\rho) \in L^{2}(Q) $ then,  by \ref{existence},   there exists a unique element $ c \in \mathbb{H} $ satisfying :

\begin{equation}\label{eq:stilln}
\int_{Q}\widetilde{\div}(\widetilde{u}c)\widetilde{\div}(\widetilde{u}\phi)\mathrm{d}x\mathrm{d}t  =
\int_{Q}f(\rho)\widetilde{\div}(\widetilde{u}\phi)\mathrm{d}x\mathrm{d}t  
\end{equation}

for all $ \phi \in H_{0}(u,Q,\partial Q_{-}) $ 

Let us  define, 
\begin{equation}\label{operator1}
T:\mathbb{H}\rightarrow \mathbb{H} 
\end{equation}
 such that 
\begin{equation}\label{operator}
T(\rho)=c  
\end{equation}
 thus a solution of the non linear problem (\ref{eq:stilsn}-\ref{CD}) is a fixed point of T.\\
Let $ \rho_{1} $ , $ \rho_{2} $  $\in$ $ \mathbb{H} $ and $c_{1}= T(\rho_{1}) $, $c_{2}=(T \rho_{2}).$\\
Since $ c_{1}-c_{2}=0 $ on $\partial Q_{-} $

$$  \mid c_{1}-c_{2} \mid ^{2}_{1,u}= \Vert \widetilde{\div}(\widetilde{u}( c_{1}-c_{2}))
\Vert_{L^{2}(Q)}^{2} = \int_{Q}\widetilde{\div}(\widetilde{u}( c_{1}-c_{2}))
 \widetilde{div}(\widetilde{u}( c_{1}))    \mathrm{d}x\mathrm{d}t  -\int_{Q}\widetilde{div}(\widetilde{u}( c_{1}-c_{2}))
 \widetilde{\div}(\widetilde{u}( c_{2}))    \mathrm{d}x\mathrm{d}t.
   $$

 For $c_{1}= T(\rho_{1}) $ and $c_{2}=T (\rho_{2}) $ we have 

$$\int_{Q}\widetilde{\div}(\widetilde{u}c_{1})\widetilde{\div}(\widetilde{u}( c_{1}-c_{2}))\mathrm{d}x\mathrm{d}t  =\int_{Q}f(\rho_{1})\widetilde{\div}(\widetilde{u}( c_{1}-c_{2}))\mathrm{d}x\mathrm{d}t$$
 and 
$$\int_{Q}\widetilde{\div}(\widetilde{u}c_{2})\widetilde{\div}(\widetilde{u}( c_{1}-c_{2}))\mathrm{d}x\mathrm{d}t  =\int_{Q}f(\rho_{2})\widetilde{\div}(\widetilde{u}( c_{1}-c_{2}))\mathrm{d}x\mathrm{d}t.$$ 
Then a computation yields  

$$  \mid c_{1}-c_{2} \mid ^{2}_{1,u}= \int_{Q}f(\rho_{1})
 \widetilde{\div}(\widetilde{u}( c_{1}-c_{2}))    \mathrm{d}x\mathrm{d}t  -\int_{Q}\
 f(\rho_{2})\widetilde{\div}(\widetilde{u}( c_{1}-c_{2}))    \mathrm{d}x\mathrm{d}t.
   $$

$$  \mid c_{1}-c_{2} \mid ^{2}_{1,u}= \int_{Q}(f(\rho_{1})-f(\rho_{2}))
 \widetilde{\div}(\widetilde{u}( c_{1}-c_{2}))    \mathrm{d}x\mathrm{d}t $$

By  Young's inequality, we get:
$$\mid c_{1}-c_{2} \mid ^{2}_{1,u} \leq  \Vert f(\rho_{1})-f(\rho_{2})\Vert_{ {L^{2}(Q)}} 
\Vert \widetilde{\div}(\widetilde{u}( c_{1}-c_{2})) \Vert_{ {L^{2}(Q)}}
   $$   
  
Since $f$ is $k-$ Lipschitz in $ \mathbb{R} $ and $ \mid c_{1}-c_{2} \mid_{1,u} =\Vert \widetilde{div}(\widetilde{u}( c_{1}-c_{2})) \Vert_{ {L^{2}(Q)}}  $
we have:
$$\mid c_{1}-c_{2} \mid _{1,u}^{2} \leq k \Vert \rho_{1}-\rho_{2}\Vert_{ {L^{2}(Q)}}^{2} \mid c_{1}-c_{2} \mid_{1,u} $$   

and hence  

$$\mid c_{1}-c_{2} \mid^{2}_{1,u} \leq  kc_{p} \mid \rho_{1}-\rho_{2} \mid_{1,u}   \mid c_{1}-c_{2} \mid_{1,u}. $$   
Finally we get 
$$\mid T(\rho_{1})-T(\rho_{2})  \mid_{1,u} \leq k c_{p} \mid \rho_{1}-\rho_{2} \mid_{1,u}.  $$   
Thus T is a strict contraction, provided  that  $k c_{p}<1.$ The 
Banach's fixed point theorem   ensures the existence and uniqueness  of $ c \in \mathbb{H} $ with $T(c)=c$ which solves  (\ref{eq:stilsn})-(\ref{CD})
\end{proof}

\begin{remark}
In the free-divergence case, the previous assumption gives $ k < \frac{1}{2T} $  thus we get a solution for small times. But it  can not be extended because of the lost of  continuity.\\
The constant $ c_{p} $  is not optimal (see \cite{BP} for more details). And  so, the condition $k c_{p}<1 $ could  be improved.
\end{remark}
Now, let us state and prove the following technical  lemmas  that will be key steps in the building of the  next existence theorem.
\begin{lemma}\label{lem1}
There is a  positive constant  $C>0$ such that for any   $  \phi \in D(\overline{Q})  \,\,\text{verifying  } \,\, \phi=0  \,\,\text{on  } \partial Q_{-} $, we have  $    \Vert  \widetilde{\nabla} \phi  \Vert_{L^{2}(Q)^{d+1}} \leq C \Vert  \widetilde{\div} (\widetilde{u} \phi) \Vert_{L^{2}(Q)}   $
\end{lemma}

\begin{proof}
Let us suppose that  the inequality is false. Then  for any integer  $n \in \mathbb{N}$, there is  $\phi_{n}\in D(\overline{Q}) $ tuch that :
\begin{equation}\label{eq1}
 \Vert  \widetilde{\nabla} \phi_{n}  \Vert_{L^{2}(Q)^{d+1}} > n \Vert  \widetilde{\div} (\widetilde{u} \phi_{n}) \Vert_{L^{2}(Q)}
\end{equation}
If  $n$ is such that   $\Vert  \widetilde{\nabla} \phi_{n}  \Vert_{L^{2}(Q)^{d+1}} =0 $ then
$ \Vert  \widetilde{\nabla} \phi_{n}  \Vert_{L^{2}(Q)^{d+1}} = n \Vert  \widetilde{\div} (\widetilde{u} \phi_{n}) \Vert_{L^{2}(Q)} =0$ which is a contradiction with  (\ref{eq1}).\\
Now dividing (\ref{eq1}) by $\Vert  \widetilde{\nabla} \phi_{n}  \Vert_{L^{2}(Q)^{d+1}} $, we have :\\
\begin{equation}\label{eq2}
 \Vert  \widetilde{\nabla}   \frac{\phi_{n}}{ \Vert  \widetilde{\nabla} \phi_{n}  \Vert}  \Vert_{L^{2}(Q)^{d+1}} > n \Vert  \widetilde{\div} (\widetilde{u}\frac{\phi_{n}}{ \Vert  \widetilde{\nabla} \phi_{n}  \Vert}) \Vert_{L^{2}(Q)}
\end{equation}

Setting  $\theta_{n}= \frac{\phi_{n}}{ \Vert  \widetilde{\nabla} \phi_{n}  \Vert_{L^{2}(Q)^{d+1}}} $, we obtain :\\
\begin{equation}\label{eq3}
\Vert  \widetilde{\nabla} \theta_{n}  \Vert_{L^{2}(Q)^{d+1}}=1
\end{equation}
and
\begin{equation}\label{eq4}
 \Vert  \widetilde{\div} (\widetilde{u} \theta_{n}) \Vert_{L^{2}(Q)} =  \Vert  \widetilde{\div} (\widetilde{u}\frac{\phi_{n}}{ \Vert  \widetilde{\nabla} \phi_{n}  \Vert}) \Vert_{L^{2}(Q)}.
\end{equation}
Thanks to  (\ref{eq3}) and  (\ref{eq4}), the inequality  (\ref{eq2}) can be written as follows
 \begin{equation}\label{eq5}
  \Vert  \widetilde{\div} (\widetilde{u} \theta_{n}) \Vert_{L^{2}(Q)} <\frac{1}{n}.
 \end{equation}
By curved inequality (also named curved Poincar\'e inequality)  we get existence of a positive constant  $A>0$ such that :
$$ \Vert \theta_{n}\Vert_{L^{2}(Q)} \leq \sqrt{A}  \Vert  \widetilde{\div} (\widetilde{u} \theta_{n}) \Vert_{L^{2}(Q)}, $$ then
\begin{equation}\label{eq6}
 \Vert \theta_{n}\Vert_{L^{2}(Q)} \leq  \frac{ \sqrt{A}}{n}.
\end{equation}
This implies that :
\begin{equation}\label{eq7}
 \theta_{n} \longrightarrow 0  \,\, \mbox{in} \,\, L^{2}(Q).
\end{equation}
From  (\ref{eq3}) and  (\ref{eq6}) one deduces that  $(\theta_{n})$  is bounded in  $H^{1}(Q).$ Then there is a convex combination of the sequence $(\theta_{n})$  that converges  to $\theta^{*}\in H^{1}(Q)$ weakly, and  so in  $L^{2}(Q)$ too. Using  (\ref{eq7}),  this convex combination converges to  $0$ in  $L^{2}(Q)$. Thanks to the uniqueness of the limit we have  $\theta^{*}=0$.\\
As a sum up, one sees  that (\ref{eq1}) yields  existence of a sequence   $( \theta_{n})_{n}\subset D(\overline{Q}) \subset H^{1}(Q)  $ satisfying:  

\begin{equation}\label{RSUME}
\left \{ \begin{array}{l} 
 \theta_{n} \longrightarrow 0  \,\, \mbox{weakly in } \,\, H^{1}(Q)       \,\,(i)           \\
 \Vert  \widetilde{\nabla} \theta_{n}  \Vert_{L^{2}(Q)^{d+1}}=1  \,\,\mbox{for any} \,\,  n\in \mathbb{N} \,\, (ii)
 \end{array}
\right .
\end{equation}

(i) implies that $ \widetilde{\nabla} \theta_{n}  \longrightarrow 0  \,\, \mbox{weakly in } \,\, L^{2}(Q) . $\\
Let  $\psi \in L^{2}(Q)^{d+1}$ such that $\Vert  \psi  \Vert_{L^{2}(Q)^{d+1}}=1.$ \\We have  :
$(\psi, \widetilde{\nabla} \theta_{n})  \longrightarrow 0  \,\, \mbox{in} \,\, \mathbb{R}  $.\\
The translation of the definition of the limit allows us to write:\\
 $\exists n_{0} \in \mathbb{N}$ such that for any  $n\geq n_{0}$, we have  $\mid (\psi, \widetilde{\nabla} \theta_{n})     \mid<1$.\\
Thus we get  $ \underset{\Vert \psi \Vert_{L^{2}(Q)^{d+1}}=1}{\text{Sup }} \mid (\psi, \widetilde{\nabla} \theta_{n})     \mid<1 $.\\
Hence, one deduces that $  \Vert  \widetilde{\nabla} \theta_{n}  \Vert_{ L^{2}(Q)^{d+1}}<1$ for any  $n\geq n_{0}:$  what is in contradiction with (ii).
\end{proof}
\begin{lemma}\label{lem2}
Let   $f:\mathbb{R}\longrightarrow \mathbb{R} $ be a   $k$ -Lipschitzian function.\\
For any  $\rho \in H(u,Q) ,  $we have   $ f(\rho) \in  H^{1}(Q). $   In addition there exists a positive  $C>0$ such that 
$$  \Vert \widetilde{\nabla}  f(\rho)\Vert_{L^{2}(Q)^{d+1}} \leq  C    \Vert  \widetilde{\div} (\widetilde{u} \rho) \Vert_{L^{2}(Q)}   $$  
\end{lemma}

\begin{proof}
Let  $\rho \in H(u,Q) $, then  there is a sequence  $(\rho_{n})\subset D(\overline{Q})$  that converges to  $\rho$ in  $H(u,Q)$.\\
Since  $f$ is  $k$ -Lipschitzian,   we get:\\
 $\mid f( \rho_{n}) \mid \leq k \mid \rho_{n} \mid +\mid f(0) \mid  $ and  $\mid f(\rho ) \mid \leq k\mid \rho \mid +\mid f(0) \mid  $.\\
Therefore  $( f( \rho_{n}))\subset L^{2}(Q)$ and $f(\rho)\in L^{2}(Q)$. In addition:
  $\Vert f(\rho_{n})-f(\rho)\Vert_{L^{2}(Q)} \leq k \Vert \rho_{n}-\rho\Vert _{L^{2}(Q)}  $ and  $\rho_{n}$ converges to $\rho$ in  $L^{2}$ thus   $f(\rho_{n})$ converges to  $f(\rho)$ in $L^{2}.$  And in particular  any  convex combination of   $f(\rho_{n})$ converges to $f(\rho)$ in $L^{2}.$\\
 
Now let us take , $ x,y $ in  $Q$  .\\
\begin{equation}\label{fond}
\mid f(\rho_{n}(x))-f(\rho_{n}(y))\mid \leq k \mid \rho_{n}(x)-\rho_{n}(y) \mid  
\end{equation}
$$ \mid f(\rho_{n}(x))-f(\rho_{n}(y))\mid \leq k \mid  \nabla \rho_{n} \mid_{\infty} \mid x-y\mid.$$
Under Rademacher's theorem,  for any integer    $n$, the function $f(\rho_{n})$  is differentiable almost everywhere  and there is  a positive constant depending on $n,$  $C
_{n}$ such that $\mid \frac{\partial f(\rho_{n})}{\partial x_{i}} \mid \leq C_{n} ;$  and then  $ \frac{\partial f(\rho_{n})}{\partial x_{i}} \in L^{2}(Q)  \,\,\mbox{for any  } \,\, i=1,...,d+1$ . \\
Using  again the inequality  (\ref{fond}), one sees that :\\
$$\mid \frac{\partial f(\rho_{n})}{\partial x_{i}} \mid \leq k  \mid \frac{\partial \rho_{n}}{\partial x_{i}} \mid  \,\,\mbox{for any  } \,\, i=1,...,d+1 $$ 
and then 
$$  \Vert \widetilde{\nabla}  f(\rho_{n})\Vert_{L^{2}(Q)^{d+1}} \leq   k \Vert \widetilde{\nabla}  \rho_{n}\Vert_{L^{2}(Q)^{d+1}}.    $$ 
By Lemma \ref{lem1}, we have
 $    \Vert  \widetilde{\nabla} \rho_{n}  \Vert_{L^{2}(Q)^{d+1}} \leq C \Vert  \widetilde{\div} (\widetilde{u}  \rho_{n}) \Vert_{L^{2}(Q)}   $.\\
This yields
\begin{equation}\label{cible}
\Vert \widetilde{\nabla}  f(\rho_{n})\Vert_{L^{2}(Q)^{d+1}} \leq   k C \Vert  \widetilde{\div} (\widetilde{u}  \rho_{n}) \Vert_{L^{2}(Q)}. 
\end{equation}
Since  $(\rho_{n})$  convergs to  $\rho$ in  $H(u,Q)$ we get  $\Vert  \widetilde{\div} (\widetilde{u}  \rho_{n}) \Vert_{L^{2}(Q)} $  converges to
$\Vert  \widetilde{\div} (\widetilde{u}  \rho) \Vert_{L^{2}(Q)}.  $ And we can conclude that   $ (f(\rho_{n}))$ is bounded in $H^{1}(Q)$.\\
 And more, we have $ (f(\rho_{n}))$ is bounded in  $H^{1}(Q).$ Then there is  $\theta \in H^{1}(Q) $ such that $ (f(\rho_{n}))$ converges to  $\theta$ weakly. Thanks to  Mazur's lemma, there is a convex combination of the sequence $ (f(\rho_{n}))$,denoted  $\theta_{n}$ that strongly  converges to  $\theta$ in $H^{1}(Q)$and then in  $L^{2}.$  And the same convex combination converges to $f(\rho) $ in $L^{2}(Q)$ .\\
Under uniqueness in  $L^{2}(Q)$, we have  $f(\rho)=\theta$ but $\theta \in H^{1}(Q)$. This ensures us that   $f(\rho) \in H^{1}(Q)$.\\
 Passing to the limit the inequality   (\ref{cible}), yields
$$ \Vert \widetilde{\nabla}  f(\rho)\Vert_{L^{2}(Q)^{d+1}} \leq   kC  \Vert   \widetilde{\div} (\widetilde{u}  \rho) \Vert_{L^{2}(Q)}  $$
\end{proof}

\begin{lemma}\label{cle}
Let $f:\mathbb{R}\longrightarrow \mathbb{R} $  be  a $k$ -Lipschitzian function, $C$  be a  bounded subset of  $H(u,Q)$ and  $(\rho_{n})$, $(c_{n})$  be  sequences in  $C.$ Denoting by     $c$ the weak limit of   $(c_{n})$ in  $H_{0}(u,Q).$ We have  :
 $$ \int_{Q}f(\rho_{n})\widetilde{\div}(\widetilde{u}( c_{n}-c)) \mathrm{d}x\mathrm{d}t \longrightarrow 0. $$ 
 \end{lemma} 
 \begin{proof}
Since  $C\subset H(u,Q)$,   $(\rho_{n})$, $(c_{n})$ are sequences of    $C$,  there are $M>0$ and  $c\in H(u,Q) $ such that
\begin{equation}\label{fin1}
 \Vert\widetilde{\div} (\widetilde{u}  \rho_{n}) \Vert_{L^{2}(Q)}\leq M
\end{equation}
and
\begin{equation}\label{fin2}
c_{n} \rightharpoonup c \,\, \text{faiblement \,\,dans } H(u,Q) 
\end{equation}
Using the curved Poincar\'e inequaity, (\ref{fin1}) we have
 \begin{equation}\label{fin3}
 \Vert \rho_{n} \Vert_{L^{2}(Q)}\leq M\sqrt{A}
\end{equation}
\begin{equation}\label{fin4}
\mid f( \rho_{n}) \mid \leq k \mid \rho_{n} \mid +\mid f(0) \mid. 
\end{equation}
Thus  (\ref{fin3}-\ref{fin4})yield a constant  $C_{\ref{fin5}}$ such that :
\begin{equation}\label{fin5}
 \Vert f(\rho_{n}) \Vert_{L^{2}(Q)}  \leq C_{\ref{fin5}}.
\end{equation}
In another way, by Lemma\ref{lem2}, there exists a constant  $C>0$ such that 
\begin{equation}\label{fin6}
 \Vert \widetilde{\nabla}  f(\rho_{n})\Vert_{L^{2}(Q)^{d+1}} \leq  C    \Vert  \widetilde{\div} (\widetilde{u} \rho_{n}) \Vert_{L^{2}(Q)}.
\end{equation}
 From (\ref{fin1})and  (\ref{fin6}) we have the following estimation
\begin{equation}\label{fin7}
 \Vert \widetilde{\nabla}  f(\rho_{n})\Vert_{L^{2}(Q)^{d+1}} \leq  CM.
\end{equation}
Relations (\ref{fin5}) et (\ref{fin7}) imply that the sequence  $(  f(\rho_{n}))$ is bounded in  $H^{1}(Q).$ Then, by Rellich's theorem,  
even if it means extracting a subsequence, there is  $F\in L^{2}(Q)$ such that 
\begin{equation}\label{fin8}
  f(\rho_{n}) \longrightarrow F  \,\, \mbox{strongly \,\,in }  L^{2}(Q).
\end{equation}
From (\ref{fin2}) and  (\ref{fin8}), we get 
\begin{equation}\label{fin}
\langle  f(\rho_{n}) ,\widetilde{\div} (\widetilde{u}  (c_{n}-c)) \rangle \longrightarrow (F,0)=0.
\end{equation}
Finally we  have 
$$ \int_{Q}f(\rho_{n})\widetilde{\div}(\widetilde{u}( c_{n}-c)) \mathrm{d}x\mathrm{d}t \longrightarrow 0. $$

\end{proof}
Having at hands these lemmas and  using  fixed Schauder's theorem,  we can proceed further to get  existence  and uniqueness results.

\begin{theorem}
Let  $ u\in L^{\infty}(Q) $ with $ \div(u)\in L^{\infty}(Q) $, and $ c_{b} \in G_{-},$ 
 $ f \in W^{1,\infty} (\mathbb{R}).$ Then the problem (\ref{eq:stilsn})-(\ref{CD}) has a solution in $ H_{0}(u,Q,\partial Q_{-} ) $.\\
\end{theorem}
\begin{proof}
Since $ c_{b} \in G_{-} $ changing the source term if necessary, we shall assume that $ c_{b}=0 \,\text{on} \, \partial Q_{-} $.\\
Existence.\\
The proof is relied mainly on the Schauder's fixed theorem.\\
Step 1: we first have to choose  a bounded subset  $\mathbb{X}$ of $ H_{0}(u,Q,\partial Q_{-})$ and a mapping  $ T :  \mathbb{X} \rightarrow \mathbb{X} $.
To achieve this aim, for all $ \rho \in V $,  under the Lemma\ref{lem2},  or  since $ f \in W^{1,\infty }(\mathbb{R}) $ we have
$ f(\rho) \in L^{2}(Q).$ Then by  Theorem \ref{existence} there exists a function $c \in H_{0}(u,Q,\partial Q_{-}) $ such that

 $$ \int_{Q}\widetilde{\div}(\widetilde{u}c)\widetilde{\div}(\widetilde{u}\phi)\mathrm{d}x\mathrm{d}t 
=\int_{Q}f(\rho)\widetilde{\div}(\widetilde{u}\phi)\mathrm{d}x\mathrm{d}t  \,\, \text{for all } \, \phi \in H_{0}(u,Q,\partial Q_{-}).$$
Moreover $ \mid c\mid_{1,v} \leq \Vert f(\rho) \Vert_{L^{2}(Q)}.$\\
 Since  $ f \in W^{1,\infty} (\mathbb{R})$, we have $  \mid c \mid_{1,u} \leq  \mid f \mid_{L^{\infty}}\mid Q \mid^{\frac{1}{2}}.$

Let us define,
$ T :  H_{0}(u,Q,\partial Q_{-}) \rightarrow H_{0}(u,Q,\partial Q_{-}) $ such that
$ c=T(\rho) $.\\
Solving (\ref{stilpnnl}) is equivalent  to show the existence of fixed point theorem of T.\\
Let us proceed further and choose a convex set $ \mathbb{X} $ as follows :
$$ \mathbb{X} =\{\phi \in H_{0}(u,Q,\partial Q_{-}) ,\mid \phi \mid_{1,u} \leq M \}   $$
when $M$ is to be precised later. 
 $$ \mid T\rho \mid_{1,u}=  \mid c \mid_{1,u} \leq  \mid f \mid_{L^{\infty}}\mid Q \mid^{\frac{1}{2}} ,\, \text{for all } \rho \in \mathbb{X}.$$ 

Thus, choosing $ M=  \mid f \mid_{L^{\infty}}\mid Q \mid^{\frac{1}{2}},$ the following inclusion yields 
$$ T(H_{0}(u,Q,\partial Q_{-})) \subset \mathbb{X} $$ 
and then  
$$ T(\mathbb{\mathbb{X}}) \subset \mathbb{X}.$$ 
So we will consider $ T :  \mathbb{X} \rightarrow \mathbb{X}.$

Step 2: Thus T is  continuous.\\
\begin{proof} of the step 2:
Then a  computation yields  
$$  \mid c_{1}-c_{2} \mid ^{2}_{1,u}= \int_{Q}f(\rho_{1})
 \widetilde{\div}(\widetilde{u}( c_{1}-c_{2}))    \mathrm{d}x\mathrm{d}t  -\int_{Q}\
 f(\rho_{2})\widetilde{\div}(\widetilde{u}( c_{1}-c_{2}))    \mathrm{d}x\mathrm{d}t
   $$

$$  \mid c_{1}-c_{2} \mid ^{2}_{1,u}= \int_{Q}(f(\rho_{1})-f(\rho_{2}))
 \widetilde{\div}(\widetilde{u}( c_{1}-c_{2}))    \mathrm{d}x\mathrm{d}t.$$

By  Young's iniquality, we get:
$$\mid c_{1}-c_{2} \mid ^{2}_{1,u} \leq  \Vert f(\rho_{1})-f(\rho_{2})\Vert_{ {L^{2}(Q)}} 
\Vert \widetilde{\div}(\widetilde{u}( c_{1}-c_{2})) \Vert_{ {L^{2}(Q)}}.
   $$   
  
Since $ f \in W^{1,\infty} (\mathbb{R})$, we have:
$$\Vert f(\rho_{1})-f(\rho_{2})\Vert_{ {L^{2}(Q)}}  \leq  \mid f ^{'}\mid_{L^{\infty}}\mid \Vert \rho_{1}-\rho_{2}\Vert_{ {L^{2}(Q)}}^{2}  $$   

and hence  
$$\mid c_{1}-c_{2} \mid^{2}_{1,u} \leq  \mid f ^{'}\mid_{L^{\infty}} c_{p} \mid \rho_{1}-\rho_{2} \mid_{1,u}   \mid c_{1}-c_{2} \mid_{1,u}; $$   
finally we get 
$$\mid T\rho_{1}-T\rho_{2}  \mid_{1,u} \leq  \mid f ^{'}\mid_{L^{\infty}}c_{p} \mid \rho_{1}-\rho_{2} \mid_{1,u}.$$ 

Thus T is Lipschitz so continuous.
 \end{proof}

 Steep 3: $\mathbb{X}$ is a subset convex, closed in $H_{0}(u,Q,\partial Q_{-})$   and  $T(\mathbb{X})$ compact in $L^2(Q)$.\\
\begin{proof} of step 3 :
it is clear that $\mathbb{X}$ is a convex and closed in  $H_{0}(u,Q,\partial Q_{-}).$  \\
Let $ (c_{n}) $ be sequences in $T(\mathbb{X})$, then there exists $(\rho_{n})$ sequence in $H_{0}(u,Q,\partial Q_{-})$ such that 
\begin{equation}\label{eq32}
\int_{Q}\widetilde{\div}(\widetilde{u}c_{n})\widetilde{\div}(\widetilde{u}\phi)\mathrm{d}x\mathrm{d}t  =
\int_{Q}f(\rho_{n})\widetilde{\div}(\widetilde{u}\phi)\mathrm{d}x\mathrm{d}t \,\, \forall \, \phi \in   H_{0}(v,Q,\partial Q_{-}).
\end{equation}

Since $(c_{n})$ bounded in $H_{0}(u,Q,\partial Q_{-})$ then there exist $c\in H_{0}(u,Q,\partial Q_{-})$ such that 
 $$ c_{n} \rightharpoonup c \,\, \text{weakly \,\,in }\,\, H_{0}(u,Q,\partial Q_{-}) $$
then  $ \widetilde{div}(\widetilde{u} (c_{n}-c)) \rightharpoonup 0  \,\, \text{weakly \,\,in } L^{2}(Q) $  
, in particular 
\begin{equation}\label{conv1}
\int_{Q}\widetilde{\div}(\widetilde{u}( c_{n}-c))  \widetilde{\div}(\widetilde{u}( c))\mathrm{d}x\mathrm{d}t \longrightarrow 0
\end{equation}
 
\begin{equation}
  \mid c_{n}-c \mid ^{2}_{1,u}= \Vert \widetilde{\div}(\widetilde{u}( c_{n}-c))
\Vert_{L^{2}(Q)}^{2} = \int_{Q}\widetilde{\div}(\widetilde{u}( c_{n}-c))
 \widetilde{div}(\widetilde{u}( c_{n}))    \mathrm{d}x\mathrm{d}t  -\int_{Q}\widetilde{div}(\widetilde{u}( c_{n}-c))
 \widetilde{\div}(\widetilde{u}( c))    \mathrm{d}x\mathrm{d}t.
\end{equation} 
Using \ref{eq32}, we have 

$$  \mid c_{n}-c \mid ^{2}_{1,u} = \int_{Q}f(\rho_{n})\widetilde{\div}(\widetilde{u}( c_{n}-c)) \mathrm{d}x\mathrm{d}t  -\int_{Q}\widetilde{div}(\widetilde{u}( c_{n}-c))
 \widetilde{\div}(\widetilde{u}( c))    \mathrm{d}x\mathrm{d}t.
$$
 
And by the Lemma \ref{cle}, even if it means extracting a subsequence,  we have 
  \begin{equation}\label{conv2}
 \int_{Q}f(\rho_{n})\widetilde{\div}(\widetilde{u}( c_{n}-c)) \mathrm{d}x\mathrm{d}t \longrightarrow 0 
\end{equation}

(\ref{conv1}) and (\ref{conv2}) imply that 
 $$  \mid c_{n}-c \mid ^{2}_{1,u} \longrightarrow 0.$$

\end{proof}

Since $ \mathbb{X} $ convex, closed  in $ H_{0}(u,Q,\partial Q_{-}) $ and  $ T :  \mathbb{X} \rightarrow \mathbb{X} $ continuous which $T(\mathbb{X})$ is relatively compact in $ H_{0}(u,Q,\partial Q_{-}) $.   By Schauder's  theorem T has a fixed point.\\


\end{proof}

\section{Existence and uniqueness result for the penalization version }
Let us consider  the space 
$$ (i)\,\, \mathbb{V}=H_{0}(u,Q,\partial Q_{-})\cap H^{1}(Q)$$
 where $ H^{1}(Q) $ is the usual Sobolev spaces,\\
with the norm\\
$$ (ii)\,\, \Vert \phi \Vert _{\mathbb{V}}^{2}= \Vert \phi \Vert_{L^{2}(Q)} ^{2}+\Vert\widetilde{div}(\widetilde{u}\phi)  \Vert_{L^{2}(Q)} ^{2} + \Vert\widetilde{\nabla} \phi \Vert_{L^{2}(Q)} ^{2}.$$
From  the curved inequality(\ref{constp}), one deduces that  the following  semi norm 
$$  \mid \phi \mid _{\mathbb{V}}=(\Vert\widetilde{div}(\widetilde{u}\phi)  \Vert_{L^{2}(Q)} ^{2} + \Vert\widetilde{\nabla} \phi \Vert_{L^{2}(Q)} ^{2}    \big )^{\frac{1}{2}}  $$ becomes a norm,   equivalent to the norm given on  $ \mathbb{V} $. And  the space $ \mathbb{V}$ will be  equipped with the norm $\mid . \mid _{\mathbb{V}}.$

For any , $\lambda \in \mathbb{R}_{+} $ and $ f \in L^{2}(Q),$ we are going to study  the following optimization problem 


\begin{equation}\label{eq:lp}
\rho_{\lambda} =\underset{c \in \mathbb{V}}{\text{Argmin}} \ J(c)+\lambda   \Vert\widetilde{\nabla} c \Vert_{L^{2}(Q)} ^{2} =\underset{c \in \mathbb{V}}{\text{Argmin}}  \ J_{\lambda}(c)
\end{equation}
where  
$$ J(c)=\frac{1}{2} \big ( \int_{Q}(\widetilde{\div}(\widetilde{u}c)-f)^{2}\mathrm{d}x\mathrm{d}t.      \big) $$

\begin{proposition} \label{th}
For any non  negative real number  $\lambda $  and $ f \in L^{2}(Q) $, the problem (\ref{eq:lp}) has a unique solution.\\
Moreover, for any $\lambda \geq 1 $ there exists $ \alpha :=\alpha(\lambda)$ such that $  \mid c \mid _{\mathbb{V}} \leq \alpha  \Vert f \Vert_{L^{2}(Q)}   $

\end{proposition}

\begin{proof}
Since $ J_{\lambda} $ is strictly convex and Gâteaux-differentiable, we have to show that there is a function $ c \in \mathbb{V} $ such that
$ DJ_{\lambda}(c).\phi=0 $ for all $ \phi \in \mathbb{V} $

An easy computation gives 
\begin{equation}\label{stilp}
DJ_{\lambda}(c).\phi= \int_{Q}(\widetilde{\div}(\widetilde{u}c)-f)\widetilde{\div}(\widetilde{u}\phi)\mathrm{d}x\mathrm{d}t +\lambda \int_{Q}\widetilde{\nabla}c\widetilde{\nabla} \phi \mathrm{d}x\mathrm{d}t.\end{equation}
And we obtain the following weak formulation:

\begin{equation}\label{stilpn}
\int_{Q}\widetilde{\div}(\widetilde{u}c)\widetilde{\div}(\widetilde{u} \widetilde{\nabla}\phi)\mathrm{d}x\mathrm{d}t  + \lambda \int_{Q}\widetilde{\nabla}c\widetilde{\nabla}\phi\mathrm{d}x\mathrm{d}t=\int_{Q}f \widetilde{\div}(\widetilde{u}\phi)\mathrm{d}x\mathrm{d}t
\end{equation}

for all $ \phi \in \mathbb{V}.$

Let us now consider the bilinear form $ a_{\lambda}(.,.):\mathbb{V} \times \mathbb{V} \rightarrow \mathbb{R}$ defined for all $ \phi \ , \psi \in \mathbb{V} $ by :

$$ 
a_{\lambda}(\phi,\psi)=\int_{Q}\widetilde{\div}(\widetilde{u}\phi)\widetilde{\div}(\widetilde{u}\psi)\mathrm{d}x\mathrm{d}t  
+\lambda \int_{Q}\widetilde{\nabla}\phi\widetilde{\nabla}\psi\mathrm{d}x\mathrm{d}t
 $$
and the linear form $ L: \mathbb{V} \rightarrow \mathbb{R}$ defined for all $ \phi \in \mathbb{V} $ by :
$$ L(\phi)= \int_{Q}f\widetilde{\div}(\widetilde{u}\phi)\mathrm{d}x\mathrm{d}t.$$
Thus  the expression  (\ref{eq:lp}) can be written as  follows  :
find $ c \in \mathbb{V} $ such that
$$a_{\lambda}(c,\phi)=L(\phi) \ \text{for all } \ 	\phi \in \mathbb{V}  $$
Taking $m=\min(\lambda,1) > 0,$ we have 

$$ 
a_{\lambda}(\phi,\phi)=\int_{Q}\widetilde{\div}(\widetilde{u}\phi)^{2}\mathrm{d}x\mathrm{d}t  
+\lambda \int_{Q}\mid\widetilde{\nabla}\phi\mid^{2}\mathrm{d}x\mathrm{d}t \geq m  \mid \phi \mid _{\mathbb{V}}^{2}. $$

Then  $a(.,.)_{\lambda}$ is $\mathbb{V}$ elliptic on the one hand.\\
On the other hand, by using Holder's inequality we have  

$$ 
\mid a_{\lambda}(\phi,\psi) \mid\leq \Vert\widetilde{div}(\widetilde{u}\phi)  \Vert_{L^{2}(Q)} \Vert\widetilde{div}(\widetilde{u}\psi)  \Vert_{L^{2}(Q)} +\lambda \Vert \nabla \phi \Vert_{L^{2}(Q)}  \Vert \nabla \psi \Vert_{L^{2}(Q)}.$$
And the following estimate holds 
 
$$ 
\mid a_{\lambda}(\phi,\psi) \mid  \leq \max(\lambda,1) (\Vert\widetilde{div}(\widetilde{u}\phi)  \Vert_{L^{2}(Q)} \Vert\widetilde{div}(\widetilde{u}\psi)  \Vert_{L^{2}(Q)} + \Vert \nabla \phi \Vert_{L^{2}(Q)}  \Vert \nabla \psi \Vert_{L^{2}(Q)}).$$
By taking $C=\max(\lambda,1)  $ and using Cauchy-Schwarz's inequality in $ \mathbb{R}^{2} $, we have 

$$ 
\mid a_{\lambda}(\phi,\psi) \mid  \leq C  \mid \phi \mid _{\mathbb{V}} \mid \psi \mid _{\mathbb{V}}  , \, \text{for all }\, \phi ,\psi \in \mathbb{V}.$$

And we conclude that   $a_{\lambda}(.,.)$ is continuous  .\\
Let us now prove that  $ L $ is continuous . 
$$ \mid L(\phi) \mid  \leq \Vert f \Vert_{L^{2}(Q)} \Vert\widetilde{div}(\widetilde{u}\phi)\Vert_{L^{2}(Q)} $$
so,
$$\mid L(\phi)\mid \leq \Vert f \Vert_{L^{2}(Q)} \mid \phi \mid _{\mathbb{V}}.$$
  Since $L$ is linear with respect to $ \phi $, we get its continuity.\\
  Hence  by  the Lax-Milgram's theorem there is a unique solution of (\ref{eq:lp}) which satisfies 
  $$ \min(1,\lambda)\mid c_{\lambda} \mid _{\mathbb{V}}^{2} \leq a_{\lambda}(c_{\lambda},c_{\lambda})= \mid L(c_{\lambda})\mid \leq \Vert f \Vert_{L^{2}(Q)} \mid c_{\lambda} \mid _{\mathbb{V}}.$$

So for $\lambda \geq 1 $; we get the desired result  $  \mid c \mid _{\mathbb{V}} \leq   \Vert f \Vert_{L^{2}(Q)}.$

  \end{proof}

\begin{theorem}
Let $\lambda > 1 $ and $ f \in W^{1,\infty} (\mathbb{R})$.Then there exists  function $ c_{\lambda} \in \mathbb{V}$ such that 
\begin{equation}\label{stilpnnl}
\int_{Q}\widetilde{\div}(\widetilde{u}c_{\lambda})\widetilde{\div}(\widetilde{u}\phi)\mathrm{d}x\mathrm{d}t 
+\lambda\int_{Q}\widetilde{\nabla}c_{\lambda}\widetilde{\nabla}\phi \mathrm{d}x\mathrm{d}t=\int_{Q}f(c_{\lambda})\widetilde{\div}(\widetilde{u}\phi)\mathrm{d}x\mathrm{d}t  \,\, \text{for all } \, \phi \in \mathbb{V}
\end{equation}
for all $ \phi \in \mathbb{V} $.\\
The solution  is unique if  $ \lambda > 2T^{2} \mid f'\mid_{L^{\infty}(\mathbb{R})}^{2}\Vert\widetilde{u}  \Vert _{ {L^{2}(Q)}}^{2} $ and $ \div(u)=0.$
\end{theorem}
\begin{proof}\\
\textbf{A-Existence:}\\
The proof is relied mainly on the Schauder's fixed theorem.\\
Step 1: we first have to choose  a bounded subset  $\mathbb{X}$ of $\mathbb{V}$ and   a mapping  $ T :  \mathbb{X} \rightarrow \mathbb{X} $.
To achieve this aims, for all $ \rho \in V $ , since $ f \in W^{1,\infty }(Q) $,
$ f(\rho) \in L^{2}(Q) $, then by  Proposition \ref{th} there exists a function $ c_{\lambda} \in \mathbb{V} $ such that

 $$ \int_{Q}\widetilde{\div}(\widetilde{u}c_{\lambda})\widetilde{\div}(\widetilde{u}\phi)\mathrm{d}x\mathrm{d}t 
+\int_{Q}\widetilde{\nabla}c_{\lambda}\widetilde{\nabla}  \phi \mathrm{d}x\mathrm{d}t=\int_{Q}f(\rho)\widetilde{\div}(\widetilde{u}\phi)\mathrm{d}x\mathrm{d}t  \,\, \text{for all } \, \phi \in \mathbb{V}.$$
Moreover $ \mid c_{\lambda}\mid_{\mathbb{V}} \leq \Vert f(\rho) \Vert_{L^{2}(Q)}.$\\
Since  $ f \in W^{1,\infty} (\mathbb{R})$, we have $  \mid c_{\lambda} \mid_{\mathbb{V}} \leq  \mid f \mid_{L^{\infty}}\mid Q \mid^{\frac{1}{2}}.$

Let us define,
$ T :  \mathbb{V} \rightarrow \mathbb{V} $ such that
$ c_{\lambda}=T(\rho) $.\\
Solving (\ref{stilpnnl}) is equivalent  to show the existence of fixed point theorem of T.\\
Let us proceed further and choose a convex set $ \mathbb{X} $ as follows :
$$ \phi \in \mathbb{V},\mid \phi \mid_{\mathbb{V}} \leq M    $$
when $M$ is to be precised later. 
And  $$ \mid T\rho \mid_{\mathbb{V}}=  \mid c_{\lambda} \mid_{\mathbb{V}} \leq  \mid f \mid_{L^{\infty}}\mid Q \mid^{\frac{1}{2}} ,\, \text{for all } \rho \in \mathbb{X} $$ 

Thus, choosing $ M=  \mid f \mid_{L^{\infty}}\mid Q \mid^{\frac{1}{2}}  $, the following inclusion yields 
$$ T(\mathbb{V}) \subset \mathbb{X} $$ 
and then  
$$ T(\mathbb{\mathbb{X}}) \subset \mathbb{X}.$$ 
So we will consider $ T :  \mathbb{X} \rightarrow \mathbb{X} $

Step 2: T is  continuous for all $ \lambda \geq 1.$
\begin{proof} of the step 2:
T can be written as composition of following application 
$$ L^{2}(Q)\longrightarrow L^{2}(Q) \longrightarrow \mathbb{V} \hookrightarrow L^{2}(Q) $$
$$\rho  \longmapsto \widetilde{f}(\rho)=f\circ \rho \longmapsto T(\rho) \hookrightarrow T(\rho).$$
By Caratheodory theorem $\rho  \longmapsto \widetilde{f}(\rho)=f\circ \rho $  is continuous from $ L^{2}(Q) $ into $ L^{2}(Q) $. And Lax-Milgram's  lemma gives the continuity of $f\circ \rho \longmapsto T(\rho)$ from  $ L^{2}(Q) $ into $\mathbb{V} $.
Using the curved inequality(\ref{constp}),
it is easy to see  that the injection 
 $ \rho \in V \longmapsto \rho \in L^{2}(Q)  $
 is also  continuous.

Then T is continuous

 \end{proof}

Steep 3: $\mathbb{X}$ is a subset convex and compact in $L^{2}(Q)$ .\\
\begin{proof} of step 3 :

$$\Vert \phi \Vert_{H^{1}(Q)}^{2}=  \Vert \phi \Vert_{ {L^{2}(Q)}}^{2}+ \Vert \widetilde{\nabla}\phi \Vert_{ {L^{2}(Q)}}^{2}  \,\,\,\,\,\,\, \forall   \phi \in H^{1}(Q) $$
By the inequality  (\ref{constp}), we have 

$$\Vert \phi \Vert_{H^{1}(Q)}^{2} \leq  (1+c_{p}^{2})(\Vert \widetilde{\div}(\widetilde{u}\phi) \Vert_{ {L^{2}(Q)}}^{2}   + \Vert \widetilde{\nabla}\phi \Vert_{ {L^{2}(Q)}}^{2})= (1+c_{p}^{2}) \mid \phi \mid_{\mathbb{V}}  \,\,\,\,\,\,\, \forall \phi \in \mathbb{V}.$$
Then $ \mathbb{X} $ which is bounded in $ \mathbb{V} $  is bounded in $H^{1}(Q). $ 
And by Rellich's  theorem ,we know that $H^{1}(Q) \subset L^{2}(Q)$ with compact injection so 
 $ \mathbb{X} $ is relatively compact in $ L^{2}(Q). $\\
Moreover $ \mathbb{X} $ is closed in $ L^{2}(Q).$\\
In fact  let $ x_{n} $ be sequence in $ \mathbb{X} $ with $ x_{n}\longrightarrow x \in L^{2}(Q) $ ,then $ x_{n} $ is bounded in $ \mathbb{V} $ which is a  reflexive Banach space then there is a subsequence $ x_{nk} $ that converges in the weak topology $ \sigma(\mathbb{V},\mathbb{V}^{*}) $ to $ x^{*}\in \mathbb{V}.$\\
$ \mathbb{X} $ is convex closed in the strong topology then $ \mathbb{X} $ is convex closed in the weak  topology (see \cite{HB},Theorem 3.2) , so we have $ x^{*}\in \mathbb{X} $.

And from Mazur's theorem ,there are convex combination  of  $ x_{nk} $ ,themselves elements of 
$ \mathbb{X} $ which converge strongly towards $ x^{*}\in \mathbb{X} $.\\
But these same convex combinations converge towards $ x \in \mathbb{X} $ in $ L^{2}(Q) $ .
By uniqueness  of the limit in $ L^{2}(Q) $ ,we have $ x=x^{*} $ .

Furthermore,  $$\mid v  \mid_{\mathbb{V}} \leq \liminf \mid x_{nk}  \mid_{\mathbb{V}} \leq 	M  \, \text{a.e} \,  x \in\mathbb{X}; $$
therefore $ \mathbb{X} $ is closed  in $ L^{2}(Q) $.\\
Since  $ \mathbb{X} $ is relatively compact and closed  in $ L^{2}(Q) $ then it  is compact in 
$ L^{2}(Q) $.
\end{proof}

Since $ \mathbb{X} $ is  convex, compact in $ L^{2}(Q) $ and      $ T :  \mathbb{X} \rightarrow \mathbb{X} $ continuous, from Schauder's fixed point theorem T has a fixed point.\\

\textbf{B-Uniqueness:} \\
Let $ \rho_{\lambda} $ and $  \overline{\rho_{\lambda}}  $ be two solutions of \ref{stilpnnl} ,we have 

$$ \int_{Q}\mid\widetilde{\div}(\widetilde{u}(\rho_{\lambda} -\overline{\rho_{\lambda}}))\mid^{2}\mathrm{d}x\mathrm{d}t +\lambda\int_{Q}\mid \widetilde{\nabla}(\rho_{\lambda} -\overline{\rho_{\lambda}})\mid^{2} \mathrm{d}x\mathrm{d}t=\int_{Q}(f(\rho_{\lambda})-f(\overline{\rho_{\lambda}}))\widetilde{\div}(\widetilde{u}(\rho_{\lambda} -\overline{\rho_{\lambda}}))\mathrm{d}x\mathrm{d}t.$$
By Young's inequality, we have
$$2\Vert\widetilde{\div}(\widetilde{u}(\rho_{\lambda} -\overline{\rho_{\lambda}})) \Vert_{ {L^{2}(Q)}}^{2}+2\lambda  \Vert \widetilde{\nabla}(\rho_{\lambda} -\overline{\rho_{\lambda}}) \Vert _{ {L^{2}(Q)}}^{2} \leq \Vert (f(\rho_{\lambda})-f(\overline{\rho_{\lambda}})) \Vert _{ {L^{2}(Q)}}^{2}+\Vert \widetilde{\div}(\widetilde{u}(\rho_{\lambda} -\overline{\rho_{\lambda}}))\Vert _{ {L^{2}(Q)}}^{2}  $$

Since $ f \in W^{1,\infty} (\mathbb{R}) $ we have  
$$ \Vert f(\rho_{\lambda})-f(\overline{\rho_{\lambda}}) \Vert _{ {L^{2}(Q)}} \leq \mid f'\mid_{L^{\infty}(\mathbb{R})}\Vert \rho_{\lambda}-\overline{\rho_{\lambda}} \Vert _{ {L^{2}(Q)}} $$
and it follows that

$$2\Vert\widetilde{\div}(\widetilde{u}(\rho_{\lambda} -\overline{\rho_{\lambda}})) \Vert_{ {L^{2}(Q)}}^{2}+2\lambda  \Vert \widetilde{\nabla}(\rho_{\lambda} -\overline{\rho_{\lambda}}) \Vert _{ {L^{2}(Q)}}^{2} \leq  \mid f'\mid_{L^{\infty}(\mathbb{R})}^{2}\Vert \rho_{\lambda}-\overline{\rho_{\lambda}} \Vert _{ {L^{2}(Q)}}^{2}  +\Vert \widetilde{\div}(\widetilde{u}(\rho_{\lambda} -\overline{\rho_{\lambda}}))\Vert _{ {L^{2}(Q)}}^{2}.$$

Since $div(u)=0$ , remark \ref{rq1} yields  $$ \Vert \rho_{\lambda}-\overline{\rho_{\lambda}} \Vert _{ {L^{2}(Q)}}^{2} \leq 4T^{2} \Vert  (\widetilde{u}, \widetilde{\nabla}( \rho_{\lambda}-\overline{\rho_{\lambda}}))  \Vert _{ {L^{2}(Q)}}^{2}.$$
And then, we have 

$$2\Vert\widetilde{\div}(\widetilde{u}(\rho_{\lambda} -\overline{\rho_{\lambda}})) \Vert_{ {L^{2}(Q)}}^{2}+2\lambda  \Vert \widetilde{\nabla}(\rho_{\lambda} -\overline{\rho_{\lambda}}) \Vert _{ {L^{2}(Q)}}^{2} \leq 4T^{2} \mid f'\mid_{L^{\infty}(\mathbb{R})}^{2}\Vert  (\widetilde{u}, \widetilde{\nabla}( \rho_{\lambda}-\overline{\rho_{\lambda}}))  \Vert _{ {L^{2}(Q)}}^{2}  +\Vert \widetilde{\div}(\widetilde{u}(\rho_{\lambda} -\overline{\rho_{\lambda}}))\Vert _{ {L^{2}(Q)}}^{2}.$$

By using Cauchy-Schwarz's inequality, we have 
$$2\Vert\widetilde{\div}(\widetilde{u}(\rho_{\lambda} -\overline{\rho_{\lambda}})) \Vert_{ {L^{2}(Q)}}^{2}+2\lambda  \Vert \widetilde{\nabla}(\rho_{\lambda} -\overline{\rho_{\lambda}}) \Vert _{ {L^{2}(Q)}}^{2} \leq 4T^{2} \mid f'\mid_{L^{\infty}(\mathbb{R})}^{2}  \Vert\widetilde{u}  \Vert _{ {L^{2}(Q)}}^{2}  \Vert  \widetilde{\nabla}( \rho_{\lambda}-\overline{\rho_{\lambda}}) \Vert _{ {L^{2}(Q)}}^{2}  +\Vert \widetilde{\div}(\widetilde{u}(\rho_{\lambda} -\overline{\rho_{\lambda}}))\Vert _{ {L^{2}(Q)}}^{2}  $$
$$\Vert\widetilde{\div}(\widetilde{u}(\rho_{\lambda} -\overline{\rho_{\lambda}})) \Vert_{ {L^{2}(Q)}}^{2}+(2\lambda-4T^{2} \mid f'\mid_{L^{\infty}(\mathbb{R})}^{2}\Vert\widetilde{u}  \Vert _{ {L^{2}(Q)}}^{2} ) \Vert \widetilde{\nabla}(\rho_{\lambda} -\overline{\rho_{\lambda}}) \Vert _{ {L^{2}(Q)}}^{2} \leq   0.$$

Thus   $ \rho_{\lambda}=\overline{\rho_{\lambda}} $  provided  that  $ \lambda > 2T^{2} \mid f'\mid_{L^{\infty}(\mathbb{R})}^{2}\Vert\widetilde{u}  \Vert _{ {L^{2}(Q)}}^{2}.$


\end{proof}

\section{Numerical  study and simulations}
In this section, two numerical methods are presented for computing the solution of  semi linear conservation law  problem  (\ref{eq:stilln}). The first consist in using  Picard's  iteration or Newton-Adaptive for the  linearization of the semi linear problem.  These linearized problems are discretized by using discontinuous Galerkin's method of  the   STILS  formulation
(\ref{eq:stils})  and continuous finite element method for the penalization version  (\ref{stilpn}). Moreover, a posteriori error bounds are established when Newton iteration is used.
\\
In the sequel we shall  assume that the function $f$ is $k-$ Lipschitz then by  Rademacher's theorem (see \cite{rdm} for more details)  f is differentiable almost everywhere.  
\subsection{A  finite element method for semi linear conservations laws }
Let us assume that the problem (\ref{eq:stilsn})-(\ref{CD}) admits a unique solution $ c \in  H^{k+1}(Q)\cap H(u,Q) $.
In order to provide numerical approximation for computing the  solution of (\ref{eq:stilsn})-(\ref{CD}) after linearization,  we shall use a simple finite element approximation wich can be derived  from the use of  discontinuous Galerkin's  approximations of the space time least squares formulation. This method is introduced in \cite{Mu} for linear hyperbolic problem and \cite{MuP} for Poisson problem.\\
Let $ \mathcal{T}_{h} $ be a  regular   partition  of the domain $ Q $ more precisely a triangulation in which each element is a  polygon ( respectively  a polyhedra ) in two dimensions ( respectively in three dimensions ) . For $ k\geq 1 $, we consider the discontinuous  finite element space (see \cite{Mu} ) 
\begin{equation}
\mathcal{V}_{h}=\big\lbrace \phi \in L^{2}(Q),  \phi\mid T \in Q_{k}(T) \,\,\forall T\in \mathcal{T}_{h} \rbrace 
\end{equation}
where $ Q_{k}(T)$ is the space of linear polynomials of degree k in each variable  on $T$
and 
\begin{equation}
\mathcal{V}=\big\lbrace \phi \in L^{2}(Q),  \phi\mid T \in  H^{k+1}(T)\cap H(u,T) \,\,\forall T\in \mathcal{T}_{h} \rbrace. 
\end{equation}
It is easy to remark that $ \mathcal{V} $ contains $ \mathcal{V}_{h} $ and  $ H^{k+1}(Q)\cap H(u,Q) $.  Let   $\mathcal{E}_{h}$ be the set of all edges for  $d=1$ or flat face for $d=2$ and  $\mathcal{E}_{h}^{0}=\mathcal{E}_{h}\backslash \partial Q_{-} $ 
For $ T\in \mathcal{T}_{h} $, let  us denote by  $ h_{K} $ the diameter of $K$ and $\rho_{K}$ the supremum  of the diameters of the inscribed spheres of $K$, $h=\max h_{T}  $ the mesh size of $  \mathcal{T}_{h}$. Let us suppose that  $\mathcal{T}_{h}  $ is shape regular and also   there exists two non negative  constant $ C_{(\ref{reg})}^{(1)} $ and $ C_{(\ref{reg})}^{(2)} $  such that
\begin{equation}\label{reg}
C_{(\ref{reg})}^{(1)} \leq \frac{h_{T}}{h_{e}}\leq C_{(\ref{reg})}^{(2)} \,\, \forall \,\, T \in \mathcal{T}_{h} \,\, \forall \,\, e \subset T. 
\end{equation}
Moreover for  $T \in \mathcal{T}_{h}$, we introduce the following notations
$$\mathcal{E}_{h}(T)=\big\lbrace E \in \mathcal{E}_{h} \, ; \, E \subset \partial T \big\rbrace.  $$
For $ \phi \in \mathcal{V}_{h} $ and  $	e \in \mathcal{E}_{h}$ with $ e=\partial T _{1} \cap \partial T_{2} $, $ T_{1},T_{2} \in  \mathcal{T}_{h} $, let we define $ [\phi]  $  the jump of $\phi$ across  $e \in \mathcal{E}_{h}^{0}$ as following 
$$ [\phi] = \phi \mid_{\partial T_{1}} \widetilde{n_{1}} +\phi\mid_{\partial T_{1}}\widetilde{n_{2}} $$ and also 
$$ [ (\widetilde{u} ,\widetilde{n})\phi ] = (\widetilde{u} ,\widetilde{n_{1}})\phi\mid_{\partial T_{1}}+(\widetilde{u} ,\widetilde{n_{2}})\phi\mid_{\partial T_{2}} $$  
where $\widetilde{n_{1}}$  and $\widetilde{n_{2}}$ denote the unit  outward vectors on $\partial T_{1}$ and $\partial T_{2}$ respectively.
For $ e \in \partial Q_{-}$, $[\phi] = \phi $ and $ [ (\widetilde{u} ,\widetilde{n})\phi ] = (\widetilde{u} ,\widetilde{n})\phi.$

By considering the following bilinear form in $\mathcal{V} \times \mathcal{V} $
\begin{equation}
\mathcal{A}(c,\phi)= \sum \limits_{T \in \mathcal{T}_{h} }  \int_{T}\widetilde{\div}(\widetilde{u}c)\widetilde{\div}(\widetilde{u}\phi)\mathrm{d}x\mathrm{d}t+\sum \limits_{e \in \mathcal{E}_{h}^{0} }  \int_{e} h_{e}^{-1}  [(\widetilde{u} ,\widetilde{n}) c ][(\widetilde{u}, \widetilde{n}) \phi ] \mathrm{d}s.
\end{equation}
Since $c\in \mathcal{V} $ then 
\begin{equation}\label{sfem1}
\mathcal{A}(c,\phi)= \sum \limits_{T \in \mathcal{T}_{h} }  \int_{T}f(c) \widetilde{\div}(\widetilde{u}\phi)\mathrm{d}x\mathrm{d}t+
\sum \limits_{e \in \partial Q_{-} }  \int_{e} h_{e}^{-1}[(\widetilde{u} ,\widetilde{n}) c_{b} ][(\widetilde{u}, \widetilde{n}) \phi ] \mathrm{d}s \,\, \forall \phi \in \mathcal{V}_{h}.
\end{equation}
The corresponding  approximation of (\ref{sfem1}) is called in  (\cite{Mu})  simple finite element methods.
It is easy to see that the bilinear form  
$$ \|\phi\|_{DG}^{2}= \mathcal{A}(\phi,\phi)+ |\phi|_{\mathcal{T}_{h}, k+1}  $$ defines a norm in $\mathcal{V}$. Moreover, we have
where 
\begin{equation}
|\rho|_{\mathcal{T}_{h}, k+1}= \sum \limits_{T \in \mathcal{T}_{h} } | \rho|_{k+1,T}^{2}
\end{equation}
\begin{equation}\label{dji}
\mathcal{A}(\phi,\psi)\leq \|\phi\|_{DG} \|\psi\|_{DG} \forall \,\, \psi \,\, ,\phi \in \mathcal{V}.
\end{equation}
As in \cite{{ANRM}}, we shall use the following  abbreviation $ x\preceq y $ for signifying $ x \leq Cy$ for some constant $ C>0 $ independent to the mesh size $ h $ and $ \lambda $. 
Let $ P_{h} $ be the $ L^{2} $ projection onto  $\mathcal{V}_{h}$ , we have the following results see \cite{RIV} for more details.\\
There exists a constant $ C_{(\ref{qin})}>0 $ such that for all $ \rho \in \mathcal{V} $
\begin{equation}\label{qin}
\|\widetilde{\nabla}(\rho-P_{h}\rho)\|_{0,T} \leq C_{(\ref{qin})} h^{k} |\rho|_{k+1,T}
\end{equation} 
for all $ T \in \mathcal{T}_{h} $
and 
\begin{equation}\label{qim}
\|\rho-P_{h}\rho\|_{0,T} \leq  C_{(\ref{qin})} h^{k+1} |\rho|_{k+1,T}.
\end{equation} 
It is also proved in \cite{tracei} that , there exists a constant $C_{(\ref{tracei})}$ independent to the mesh size $ h $ such that for any $ T\in \mathcal{T}_{h}  $ and $ e \subset \partial T  $, we have 
\begin{equation}\label{tracei}
\|\rho\|_{e}^{2} \leq C_{(\ref{tracei})}( h^{-1}\|\rho\|_{T}^{2}+h\|\widetilde{\nabla} \rho\|_{T}^{2})
\end{equation} 
Finally we  deduce that the following approximation lemma.
\begin{lemma}\label{lamDG}
For all $ \rho \in  \mathcal{V} $ 
\begin{equation}
\|\rho-P_{h}\rho\|_{DG} \preceq h^{k} \|\rho\|_{\mathcal{T}_{h}, k+1} \,\, \forall \,\, T\in \mathcal{T}_{h}  
\end{equation}
\end{lemma}
\begin{proof}
$$\|\rho-P_{h}\rho\|_{DG}^{2}= \sum \limits_{T \in \mathcal{T}_{h} }  \int_{T}\widetilde{\div}(\widetilde{u}(\rho-P_{h}\rho))^{2}\mathrm{d}x\mathrm{d}t+\sum \limits_{e \in \mathcal{E}_{h} }  \int_{e} h_{e}^{-1}  \|[(\widetilde{u} ,\widetilde{n}) (\rho-P_{h}\rho)]\|^{2} \mathrm{d}s  $$
By theorem (\ref{trace21}),  $ (\widetilde{u},\widetilde{n}) \in L^{\infty}( \partial T) $, then it follows from (\ref{tracei}) and (\ref{reg})
\begin{equation}\label{one}
\int_{e} h_{e}^{-1}  \|[(\widetilde{u} ,\widetilde{n}) (\rho-P_{h}\rho)]\|^{2} \mathrm{d}s \leq C_{(\ref{reg})}^{(2)} h^{-1} \|(\widetilde{u},\widetilde{n})\|_{L^{\infty}(e)}  \int_{e}\|[ (\rho-P_{h}\rho)]\|^{2} \mathrm{d}s
\end{equation} 
This and  (\ref{tracei}) yield

\begin{equation}\label{one1}
\int_{e} \|[(\widetilde{u} ,\widetilde{n}) (\rho-P_{h}\rho)]\|^{2} \mathrm{d}s \leq 4 C_{(\ref{reg})}^{(2)}  \|(\widetilde{u},\widetilde{n})\|_{L^{\infty}(e)}  C_{(\ref{tracei})}( h^{-2}\|\rho-P_{h}\rho\|_{T}^{2}+\|\widetilde{\nabla} (\rho-P_{h}\rho)\|_{T}^{2})
\end{equation} 
And from (\ref{qin})and(\ref{qim}), it follows  
\begin{equation}\label{one2}
\int_{e} \|[(\widetilde{u} ,\widetilde{n}) (\rho-P_{h}\rho)]\|^{2} \mathrm{d}s \leq c_{(\ref{one2})}    h^{2k} |\rho|_{k+1,T}^{2}
\end{equation} 
where 
\begin{equation}
c_{(\ref{one2})}=C_{(\ref{tracei})}C_{(\ref{qin})}C_{(\ref{reg})}^{(2)}\|(\widetilde{u},\widetilde{n})\|_{L^{\infty}(e)}.
\end{equation}

We have also from triangular inequality 
\begin{equation}
\|\widetilde{\div}(\rho-P_{h}\rho)\|_{T}  \leq  \|(\widetilde{u}\widetilde{\nabla} (\rho-P_{h}\rho))\|_{T}+\|\\div( \widetilde{u})(\rho-P_{h}\rho)\|_{T}.
\end{equation}
Since $ \widetilde{u} \in L^{\infty}(T)$ and $ \div(\widetilde{u}) \in L^{\infty}(T)$, we get from (\ref{qin})-(\ref{qim})
\begin{equation}\label{two}
\|\widetilde{\div}(\rho-P_{h}\rho)\|_{T}  \leq C_{(\ref{qin})} \alpha_{(u,\ref{tw})} ( h^{k} |\rho|_{k+1,T})
\end{equation}
where 
\begin{equation}\label{tw}
\alpha_{(u,\ref{tw})}=\max{\{\| \div(\widetilde{u}) \|_{L^{\infty}(T)},|Q|\| \widetilde{u} \|_{L^{\infty}(T)}\}}.
\end{equation}

From (\ref{one2}) and (\ref{two}), we get the result.

\end{proof}

\subsubsection{A finite element method and Picard's iteration }
Let  $f$ be  a $k-$ Lipschitz function  in $ \mathbb{R} $ with $ k<\frac{1}{c_{p}} .$ In this case the  solution $c^{h}$  can be computed by using the Picard iteration of some  linear problem. The Picard iteration in this context is given by following scheme:\\
\begin{algo}
\begin{enumerate}
\end{enumerate}

\begin{itemize}
\item Start STILS-MT1 with some given $ C^{0}$
\item compute  $c^{h}_{n+1}$ from  $c^{h}_{n}$ such that 
\begin{equation}\label{sfem}
\mathcal{A}(c^{h}_{n+1},\phi_{h})= \sum \limits_{T \in \mathcal{T}_{h} }  \int_{T}f(c^{h}_{n}) \widetilde{\div}(\widetilde{u}\phi_{h})\mathrm{d}x\mathrm{d}t+
\sum \limits_{e \in \partial Q_{-} }  \int_{e} h_{e}^{-1}[(\widetilde{u} ,\widetilde{n}) c_{b} ][(\widetilde{u}, \widetilde{n}) \phi_{h} ] \mathrm{d}s \,\, \forall \phi_{h} \in \mathcal{V}_{h}.
\end{equation}
\end{itemize}
\end{algo}

\subsubsection{A finite element method and Newton's method}
We suppose that the problem (\ref{eq:stilsn})-(\ref{CD}) has  a unique solution $ \mathcal{V}=H^{2}(Q)\cap H(u,Q) $.Recalling (\ref{sfem}), we can write (\ref{eq:stilsn})-(\ref{CD}) as follows:
\begin{equation}\label{NLE}
\text{find} \,\, c \in \mathcal{V} \,\,\text{ such \,\, that} \,\,F(c)=0
\end{equation} 

where 
$$ F:\mathcal{V} \longrightarrow \mathcal{V}^{*}  $$
\begin{equation}\label{stilpnnl2}
\langle F(c),\phi \rangle_{\mathcal{V}^{*},\mathcal{V}}=\mathcal{A}(c,\phi)- \sum \limits_{T \in \mathcal{T}_{h} }  \int_{T}f(c) \widetilde{\div}(\widetilde{u}\phi)\mathrm{d}x\mathrm{d}t-
\sum \limits_{e \in \partial Q_{-} }  \int_{e} h_{e}^{-1}[(\widetilde{u} ,\widetilde{n}) c_{b} ][(\widetilde{u}, \widetilde{n}) \phi ] \mathrm{d}s \,\, \forall \phi \in \mathcal{V}.
\end{equation}
Given some initial guess $ c^{0},$  the classical Newton-Raphson's method for solving equation (\ref{NLE}), when $ F $ is differentiable, consists in generating a sequence of approximation that converges  in the quadratic sense, to the exact  solution as follows.
\begin{equation}\label{NW}
\left \{ \begin{array}{l} 
 c_{0} \in \mathbb{V} \\
c_{n+1}=c_{n}-F^{'}(c_{n})^{-1}.F(c_{n})  \,\,  \forall n\in \mathbb{N}^{*}. \\
\end{array}
\right .
\end{equation}
This method is known to produce a chaotic behavior when  $ c_{0} $ is far to the desired root see for instance 
 (see\cite{NBH}) for more details . In order to remedy the chaotic behavior the following Newton Damping method is proposed (see\cite{DAMPING}). In that case (\ref{NW}) is written as    

\begin{equation}\label{DAMP}
\left \{ \begin{array}{l} 
 C_{0} \in \mathbb{V} \\
c_{n+1}=c_{n}-\delta t F^{'}(c_{n})^{-1}.F(c_{n}) \,\,  \forall n\in \mathbb{N}^{*}. \\
\end{array}
\right .
\end{equation}
 
We shall use adaptive Newton-Galerkin's method, more precisely the damping parameter $ \delta t $  in (\ref{DAMP}) may be adjusted and  adapted in each iteration. For illustration of the choice of $ \delta t $,  let us define the Newton-Raphson's transform  as follows: 

$$\rho \longmapsto N_{F}(\rho):=-F^{'}(\rho)^{-1}.F(\rho).$$
By (\ref{DAMP}), we have 
$$\frac{c_{n+1}-c_{n}}{\delta _{n}}=N_{F}(c_{n}).  $$
And we remark that (\ref{DAMP}) may be seen as a forward Euler scheme of the following ordinary differential equation  
\begin{equation}\label{ODE}
\frac{\mathrm{d}}{\mathrm{d}s}\rho(s)=N_{F}(\rho(s))\,\,   \forall s \,\, \rho(0)=c_{0}.
\end{equation}
If $c_{n} \in \mathcal{V} $ for all $n \geq 1$ and $F$ is enough smooth for instance $ F^{'}(\rho)^{-1}.F(\rho)  $ exists for all $ \rho \in \mathcal{V}$ then, we obtain the solution of (\ref{ODE}) satisfies 
$$F(\rho(t))=F(\rho(0))\exp(-t), \,\, \forall \,\, t\geq 0.   $$ 
It is easy to see that,  $ F(\rho(t))\longrightarrow 0 $ as $ t\longrightarrow 0.$ \\
The adaptive Newton-Raphson (see \cite{NF}  )consists in choosing the damping parameter $ \delta t_{n} $ so that so that
the discrete forward Euler's solution fof  (\ref{DAMP}) stays reasonably close to the continuous solution of (\ref{ODE}).Finally we obtain the following algorithm, see \cite{ANRM}

\begin{algo}
Fix a tolerance $ \epsilon$

\begin{itemize}
\item[(i)] Start the Newton iteration with some initial guess $ c_{0} \in \mathbb{V} $
\item[(ii)] In each iteration steep $ n=1,2,... $ compute 
\begin{equation}
\delta t_{n}=\min(\sqrt{\frac{2\epsilon}{\Vert N_{F}(c_{n})\Vert _{\mathbb{V}}}},1  )
\end{equation}
\item[(iii)] compute $ c_{n+1} $ from \ref{DAMP} and go (ii)
\end{itemize}
\end{algo}

In the sequel, we suppose that $ f^{'}(c_{n}) $ exists for all $ n\geq 1, $ thus the sequels in \ref{DAMP} is well defined and we have 

\begin{equation}
\beta(c,\rho,\phi)=:\langle F^{'}(c)\rho,\phi \rangle_{\mathcal{V}^{*},\mathcal{V}}=\mathcal{A}(\rho,\phi)- \sum \limits_{T \in \mathcal{T}_{h} }  \int_{T}f^{'}(c)\rho (\widetilde{u}\phi)\mathrm{d}x\mathrm{d}t  \,\, \text{for all } \, \phi \in \mathcal{V}.
\end{equation}

Let us define 

$$ L(c,\phi):=\langle F(c),\phi \rangle_{\mathcal{V}^{*},\mathcal{V}} $$
with the  previous notation (\ref{DAMP}) can be written as follows :
given $c_{n} \in \mathcal{V}$, find $c_{n+1} \in \mathcal{V}$ such that
\begin{equation}\label{NEW}
\beta(c_{n},c_{n+1},\phi)=\beta(c_{n},c_{n},\phi)-\delta t_{n}L(c_{n},\phi) \,\, \text{for all } \, \phi \in \mathcal{V}.
\end{equation}
Let us know consider the following  finite element approximation  :
find $ c_{n+1}^{h} \in \mathcal{V} $ from $ c_{n}^{h} \in \mathcal{V}_{h} $  such that 
\begin{equation}\label{NEWFEM}
\beta(c_{n}^{h},c_{n+1}^{h},\phi)=\beta(c_{n}^{h},c_{n}^{h},\phi)-\delta t_{n}L(c_{n}^{h},\phi) \,\, \text{for all } \, \phi \in \mathcal{V}_{h}.
\end{equation}

 By introducing the following notation 
\begin{equation}
c_{n+1}^{(\delta t_{n},h)}:=c_{n+1}^{h}-(1-\delta t_{n})c_{n}^{h}
\end{equation}
and 
\begin{equation}
f^{\delta t_{n}}(c_{n+1}^{h}):=\delta t_{n}f(c_{n}^{h})+f^{'}(c_{n}^{h})(c_{n+1}^{h}-c_{n}^{h} )
\end{equation}
we have, from (\ref{NEW})

\begin{eqnarray*}
\sum \limits_{T \in \mathcal{T}_{h} }  \int_{T}\widetilde{\div}(\widetilde{u}c_{n+1}^{(\delta t_{n},h)})\widetilde{\div}(\widetilde{u}\phi)\mathrm{d}x\mathrm{d}t+\sum \limits_{e \in \mathcal{E}_{h}^{0} }  \int_{e} h_{e}^{-1}  [(\widetilde{u} ,\widetilde{n}) c_{n+1}^{(\delta t_{n},h)}] [ (\widetilde{u}, \widetilde{n}) \phi] \mathrm{d}s\\= \sum \limits_{T \in \mathcal{T}_{h} }  \int_{T} f^{\delta t_{n}}(c_{n+1}^{h})\widetilde{\div}(\widetilde{u}\phi)\mathrm{d}x\mathrm{d}t +\sum \limits_{e \in \partial Q_{-} }  \int_{e} \delta t_{n}h_{e}^{-1}  [(\widetilde{u} ,\widetilde{n}) c_{b}] [ (\widetilde{u}, \widetilde{n}) \phi] \mathrm{d}s \,\, \forall \, \phi \in \mathcal{V}_{h}.
\end{eqnarray*}
Let  us define the following quantities 
\begin{equation}
\alpha_{T}=\| \widetilde{\div}(\widetilde{u}c_{n+1}^{(\delta t_{n})})-f(c_{n+1}^{(\delta t_{n}, h)})\|_{0,T}\,\,
\text{and} \,\,\beta_{T}=\|f^{\delta t_{n}}(c_{n+1}^{h})- f(c_{n+1}^{\delta t_{n}, h}) \|_{0,T}
\end{equation}

\begin{equation}
 \alpha_{e}=\|[(\widetilde{u}, \widetilde{n}) c_{n+1}^{(\delta t_{n})}]\|_{0,e} \,\,
\text{and} \,\, \beta_{e}=\|[(\widetilde{u}, \widetilde{n}) c_{b}]\|_{0,e}. 
 \end{equation}
 
 We have also the following result expressed by an inequality.

\begin{theorem}

\begin{equation}\label{ress1}
\| F(c_{n+1}^{(\delta t_{n}, h)})\|_{\mathcal{V}^{*}} \preceq h^{k} \max{  ( ( \sum \limits_{T \in \mathcal{T}_{h} } \beta_{T}^{2})^{\frac{1}{2}}, \max{(( \sum \limits_{T \in \mathcal{T}_{h} }	\alpha_{T}^{2})^{\frac{1}{2}},( \sum \limits_{e \in \mathcal{E}_{h}^{0} }	h^{-1}_{e} \alpha_{e}^{2})^{\frac{1}{2}}+( \sum \limits_{e \in \mathcal{E}_{h}^{0} }	h^{-1}_{e} \beta_{e}^{2})^{\frac{1}{2}} }})). 
\end{equation} 
\end{theorem}
\begin{proof}
$$\langle F(c),\phi \rangle_{\mathcal{V}^{*},\mathcal{V}}=\langle F(c),\phi-P_{h}\phi \rangle_{\mathcal{V}^{*},\mathcal{V}}+\langle F(c),P_{h}\phi \rangle_{\mathcal{V}^{*},\mathcal{V}} $$
Since $P_{h}\in \mathcal{V} $, From (\ref{NEW}) it follows 
$$
\langle F(c_{n+1}^{(\delta t_{n}, h)}),P_{h}\phi \rangle_{\mathcal{V}^{*},\mathcal{V}}= \sum \limits_{T \in \mathcal{T}_{h} }  \int_{T}( f^{\delta t_{n}}(c_{n+1}^{h})- f(c_{n+1}^{\delta t_{n}, h}))\widetilde{\div}(\widetilde{u}P_{h}\phi)\mathrm{d}x\mathrm{d}t
$$
and it follows, from Cauchy–Schwarz inequality in $ L_{2}(T) $ and $ R^{q}$ with $q=\dim (\mathcal{T}_{h}) $ 
$$\langle F(c_{n+1}^{(\delta t_{n}, h)}),P_{h}\phi \rangle_{\mathcal{V}^{*},\mathcal{V}}\leq  \sum \limits_{T \in \mathcal{T}_{h} } \|f^{\delta t_{n}}(c_{n+1}^{h})- f(c_{n+1}^{\delta t_{n}, h}) \|_{0,T} \|P_{h}\phi\|_{0,T}$$

$$|\langle F(c_{n+1}^{(\delta t_{n}, h)}),P_{h}\phi \rangle_{\mathcal{V}^{*},\mathcal{V}} |\leq ( \sum \limits_{T \in \mathcal{T}_{h} } \|f^{\delta t_{n}}(c_{n+1}^{h})- f(c_{n+1}^{\delta t_{n}, h}) \|_{0,T}^{2})^{\frac{1}{2}}( \sum \limits_{T \in \mathcal{T}_{h} } \|P_{h}\phi\|_{0,T}^{2})^{\frac{1}{2}}.$$
Since $ P_{h} $ satisfies $ \sum \limits_{T \in \mathcal{T}_{h} }\|P_{h}\phi\|_{0,T}^{2} \leq  \sum \limits_{T \in \mathcal{T}_{h} }\|\phi\|_{0,T}^{2}  \,\, \forall \phi \in L^{2}(Q) $ (see \cite{HB}), we have
\begin{equation}
|\langle F(c_{n+1}^{(\delta t_{n}, h)}),P_{h}\phi \rangle_{\mathcal{V}^{*},\mathcal{V}} |\leq ( \sum \limits_{T \in \mathcal{T}_{h} } \beta_{T}^{2})^{\frac{1}{2}}( \sum \limits_{T \in \mathcal{T}_{h} } \|\phi\|_{0,T}^{2})^{\frac{1}{2}}
\end{equation}
And using lemma \ref{lamDG}, we have 
\begin{equation}\label{iph}
|\langle F(c_{n+1}^{(\delta t_{n}, h)}),P_{h}\phi \rangle_{\mathcal{V}^{*},\mathcal{V}} |\leq ( \sum \limits_{T \in \mathcal{T}_{h} } \beta_{T}^{2})^{\frac{1}{2}}h^{k} \|\phi\|_{\mathcal{T}_{h}, k+1}.
\end{equation}

\begin{eqnarray*}
\langle F(c_{n+1}^{(\delta t_{n}, h)}),\phi-P_{h}\phi \rangle_{\mathcal{V}^{*},\mathcal{V}} =\sum \limits_{T \in \mathcal{T}_{h} }  \int_{T}(\widetilde{\div}(\widetilde{u}c_{n+1}^{(\delta t_{n})})-f(c_{n+1}^{(\delta t_{n}, h)})))\widetilde{\div}(\widetilde{u}(\phi-P_{h}\phi))\mathrm{d}x\mathrm{d}t \\+\sum \limits_{e \in \mathcal{E}_{h}^{0} }  \int_{e} h_{e}^{-1} [(\widetilde{u} ,\widetilde{n})c_{n+1}^{(\delta t_{n})} ][(\widetilde{u}, \widetilde{n}) (\phi-P_{h}\phi)]\mathrm{d}s -\sum \limits_{e \in \partial Q_{-} }  \int_{e} h_{e}^{-1}[(\widetilde{u} ,\widetilde{n}) c_{b} ][(\widetilde{u}, \widetilde{n}) (\phi-P_{h}\phi) ] \mathrm{d}s
\end{eqnarray*}

\begin{eqnarray}\label{first}
\sum \limits_{T \in \mathcal{T}_{h} }  \int_{T}(\widetilde{\div}(\widetilde{u}c_{n+1}^{(\delta t_{n})})-f(c_{n+1}^{(\delta t_{n}, h)})\widetilde{\div}(\widetilde{u}(\phi-P_{h}\phi))\mathrm{d}x\mathrm{d}t \leq ( \sum \limits_{T \in \mathcal{T}_{h} }	\alpha_{T}^{2})^{\frac{1}{2}}( \sum \limits_{T \in \mathcal{T}_{h} } \|\widetilde{\div}(\widetilde{u}(\phi-P_{h}\phi) \|_{0,T}^{2})^{\frac{1}{2}}
\end{eqnarray}
\begin{equation}\label{second}
\sum \limits_{e \in \mathcal{E}_{h}^{0} }  \int_{e} h_{e}^{-1} [(\widetilde{u} ,\widetilde{n})c_{n+1}^{(\delta t_{n})} ][(\widetilde{u}, \widetilde{n}) (\phi-P_{h}\phi)]\mathrm{d}s \leq ( \sum \limits_{e \in \mathcal{E}_{h}^{0} }	h^{-1}_{e} \alpha_{e}^{2})^{\frac{1}{2}}( \sum \limits_{e \in \mathcal{E}_{h}^{0} }h^{-1}_{e} \|[(\widetilde{u}, \widetilde{n}) (\phi-P_{h}\phi)]\|_{0,e}^{2})^{\frac{1}{2}} 
\end{equation}
 \begin{equation}\label{tertio}
\sum \limits_{e \in \partial Q_{-} }  \int_{e} h_{e}^{-1} [(\widetilde{u} ,\widetilde{n})c_{b} ][(\widetilde{u}, \widetilde{n}) (\phi-P_{h}\phi)]\mathrm{d}s \leq ( \sum \limits_{e \in \mathcal{E}_{h}^{0} }	h^{-1}_{e} \beta_{e}^{2})^{\frac{1}{2}}( \sum \limits_{e \in \mathcal{E}_{h}^{0} }h^{-1}_{e} \|[(\widetilde{u}, \widetilde{n}) (\phi-P_{h}\phi)]\|_{0,e}^{2})^{\frac{1}{2}}.  
\end{equation} 
 
 The inequalities (\ref{first})-(\ref{tertio}) yield
  
\begin{equation}
|\langle F(c_{n+1}^{(\delta t_{n}, h)}),\phi-P_{h}\phi \rangle_{\mathcal{V}^{*},\mathcal{V}}| \leq \max{(( \sum \limits_{T \in \mathcal{T}_{h} }	\alpha_{T}^{2})^{\frac{1}{2}},( \sum \limits_{e \in \mathcal{E}_{h}^{0} }	h^{-1}_{e} \alpha_{e}^{2})^{\frac{1}{2}}+( \sum \limits_{e \in \mathcal{E}_{h}^{0} }	h^{-1}_{e} \beta_{e}^{2})^{\frac{1}{2}} })) \|\phi-P_{h}\phi\|_{DG}.
\end{equation}
Thus, it follows from lemma\ref{lamDG} 
\begin{equation}\label{iphi}
|\langle F(c_{n+1}^{(\delta t_{n}, h)}),\phi-P_{h}\phi \rangle_{\mathcal{V}^{*},\mathcal{V}}| \preceq    \max{( \sum \limits_{T \in \mathcal{T}_{h} }	\alpha_{T}^{2})^{\frac{1}{2}},( \sum \limits_{e \in \mathcal{E}_{h}^{0} }	h^{-1}_{e} \alpha_{e}^{2})^{\frac{1}{2}}+( \sum \limits_{e \in \mathcal{E}_{h}^{0} }	h^{-1}_{e} \beta_{e}^{2})^{\frac{1}{2}} })h^{k} \|\phi\|_{\mathcal{T}_{h}, k+1}.
\end{equation} 
From  (\ref{iph}) and (\ref{iphi}) we deduce 
\begin{equation}\label{ress1}
|\langle F(c_{n+1}^{(\delta t_{n}, h)}),\phi\rangle_{\mathcal{V}^{*},\mathcal{V}}| \preceq  \max{  ( ( \sum \limits_{T \in \mathcal{T}_{h} } \beta_{T}^{2})^{\frac{1}{2}}, \max{(( \sum \limits_{T \in \mathcal{T}_{h} }	\alpha_{T}^{2})^{\frac{1}{2}},( \sum \limits_{e \in \mathcal{E}_{h}^{0} }	h^{-1}_{e} \alpha_{e}^{2})^{\frac{1}{2}}+( \sum \limits_{e \in \mathcal{E}_{h}^{0} }	h^{-1}_{e} \beta_{e}^{2})^{\frac{1}{2}} }}))h^{k} \|\phi\|_{\mathcal{T}_{h}, k+1}.
\end{equation}

\end{proof}

\subsection{STILS  for semi linear conservations laws}
In the following, we assume that the problem (\ref{stilpnnl}) admits a unique solution  $ c \in \mathbb{V}:=  H_{0}(u,Q)\cap H^{k+1}(Q)$,  and   we will omit the dependency of the function according to the parameter $ \lambda $. Our aims are to give a  numerical methods for solving problem  (\ref{stilpnnl}) based on the classical finite element approximation of STILS formulation  and establish  posteriori estimations.
For this we shall  first  consider  first  a finite element approximation based in quadrilateral mesh by starting with the following  finite dimensional spaces  
\begin{equation}
V(\widehat{K})=\big\lbrace \phi \in C^{0}(Q),  \phi \mid \widehat{K} \in \widehat{Q_{k}}   \rbrace 
\end{equation}
where $\widehat{K}$ is so called  reference element and $\widehat{Q_{k}} $ is the space of polynomials of degree at most k in each variable, separately defined in $\widehat{K}.$ Let $ S $ be a class of invertible affine mapping defined on $\widehat{K}$ into $ \mathbb{R}^{d+1} $. For $K=F_{K}(\widehat{K})$ with $F_{K} \in S$, the finite element space can be defined by composition with the inverse of $F_{K} $  as follows
\begin{equation}
V(K)=\big\lbrace \rho:K \rightarrow \mathbb{R} : \rho=\widehat{\rho}\circ F_{K} \,\, \text{for some} \,\, \rho \in V(\widehat{K}) \rbrace. 
\end{equation}
Let $ \mathcal{T}_{h} $ be a triangulation of $ Q $ such that each  of its  element is transformation of $ \widehat{K}$ with some mapping in S. Thus we get the classical finite element approximation 
$$V_{h}=\big\lbrace \rho:Q \rightarrow \mathbb{R} : \rho \mid K  \in \ V(K)  \,\, \text{for all}  \,\,K  \rbrace.$$
In order to obtain the CFL condition stability of STILS-MT1 (see \cite{GM} ), we shall consider a strict rectangular mesh. 
Let 
\begin{equation}
\Pi:V\longrightarrow V_{h} \,\, \text{such that } \,\, \Pi q= q  \,\, \text{for all } \,\, q \in Q_{k} 
\end{equation}
be a linear operator.\\
Let  us recall the following approximation result  proved in \cite{PAR}, pp 103, Corollary 4.4.2.\\
\begin{lemma}
Let us suppose that $d\leq 2 $ and $ k \geq 1 $ then there exist $ C $ such that, for all $ 0 \leq m\leq k+1 $, $ c \in H^{k+1}(K) $ the following inequality holds 
\begin{equation}\label{raviart}
 \mid c-\Pi c \mid_{m,K} \leq \frac{h^{k+1}}{\rho^{m}} C_{\Pi,Q} \mid c\mid_{k+1,K}.
\end{equation}
\end{lemma}
  The above lemma provides us  existence of $ C>0 $ such that  the following inequalities hold for any $ c \in H(u,T)\cap H^{k+1}(T) $, 
\begin{equation}\label{l2}
\|c-\Pi c \|_{0,T} \leq C h^{k+1} |c|_{k+1,T} \,\, \forall \,\, T \in \mathcal{T}_{h}
\end{equation}
and 
\begin{equation}\label{h1}
\|\widetilde{\nabla}( c-\Pi c) \|_{0,T} \leq C h^{k} |c|_{k+1,T} \,\, \forall \,\, T \in \mathcal{T}_{h}.
\end{equation}
As in the proof of lemma \ref{lamDG} , there is a non negative constant $ C_{u}  $ such that :
\begin{equation}\label{divconv}
\|\widetilde{\div}(\widetilde{u}(c-\Pi c)) \|_{0,T} \leq C_{u}h^{k} |c|_{k+1,T} \,\, \forall \,\, T \in \mathcal{T}_{h}
\end{equation}

\subsubsection{STILS and Picard's itaration}
In this section, we suppose that $f$ is $k-$ Lipschitz with $k<\frac{1}{c_{p}}.$ Then the  mapping T  defined by (\ref{operator1})- (\ref{operator}) is a strict contraction and  thus  we shall use Picard's iteration algorithm for the linearization  of (\ref{eq:stilln})-(\ref{CD}).\\
The Picard's iteration in this context is given by following scheme:\\
\begin{algo}
\begin{enumerate}
\end{enumerate}

\begin{itemize}
\item Start STILS-MT1 with some given $ C^{0}$
\item Find $c^{h}_{n+1} \in \mathbb{V}_{h}$ from $c^{h}_{n}$ by the formula
\begin{equation}\label{atlispicard}
\int_{Q}\widetilde{\div}(\widetilde{u}c^{h}_{n+1})\widetilde{\div}(\widetilde{u}\phi_{h})\mathrm{d}x\mathrm{d}t+\lambda\int_{Q}\widetilde{\nabla}c_{n+1}^{h}\widetilde{\nabla}\phi \mathrm{d}x\mathrm{d}t  =\int_{Q}(f(c^{h}_{n})-\widetilde{\div}(\widetilde{u}C_{b} ))\widetilde{\div}(\widetilde{u}\phi_{h})\mathrm{d}x\mathrm{d}t  \,\, \forall \,\,  \phi_{h} \in  \mathbb{V}_{h}.
\end{equation}
\end{itemize}
\end{algo}

\subsection{STILS adaptive Newton method }

Since the problem  (\ref{stilpnnl}) has  a unique solution $ \mathbb{V}=H^{k+1}(Q)\cap H_{0}(u,Q) $. Then the problem can be written as follows 
\begin{equation}\label{NLEN}
\text{find} \,\, c \in \mathbb{V} \,\,\text{ such \,\, that} \,\,F_{\lambda}(c)=0
\end{equation} 

where 
$$ F_{\lambda}:\mathbb{V} \longrightarrow \mathcal{V}^{*} $$
such that 
\begin{equation}\label{billamda}
\langle F_{\lambda}(c),\phi \rangle_{\mathcal{V}^{*},\mathcal{V}} :=\int_{Q}\widetilde{\div}(\widetilde{u}c)\widetilde{\div}(\widetilde{u}\phi)\mathrm{d}x\mathrm{d}t-\lambda\int_{Q}\widetilde{\nabla}c\widetilde{\nabla}\phi \mathrm{d}x\mathrm{d}t  -\int_{Q}f(c)\widetilde{\div}(\widetilde{u}\phi)\mathrm{d}x\mathrm{d}t  \,\, \forall \,\,  \phi \in  \mathbb{V}.
\end{equation}

Since $f$ is differentiable, then $ F_{\lambda} $ is differentiable and we have 

\begin{equation}\label{billamda}
\beta_{\lambda}(c,\rho,\phi)=:\langle F_{\lambda}^{'}(c)\rho,\phi \rangle_{\mathcal{V}^{*},\mathcal{V}}=\int_{Q}\widetilde{\div}(\widetilde{u}\rho)\widetilde{\div}(\widetilde{u}\phi)\mathrm{d}x\mathrm{d}t-\lambda\int_{Q}\widetilde{\nabla}\rho\widetilde{\nabla}\phi \mathrm{d}x\mathrm{d}t  -\int_{Q}f^{'}(c)\rho\widetilde{\div}(\widetilde{u}\phi)\mathrm{d}x\mathrm{d}t  \,\, \forall \,\,  \phi_{h} \in  \mathbb{V}.
\end{equation}
Let us define also the following linear form in $\mathbb{V}$
\begin{equation}\label{llamda}
\L_{\lambda}(\rho,\phi)=\langle F_{\lambda}(\rho),\phi \rangle_{\mathcal{V}^{*},\mathcal{V}}. 
\end{equation}
We assume that $F$ invertible ,  inserting (\ref{billamda})and  (\ref{llamda}) in (\ref{DAMP}) we get
\begin{equation}\label{NEW2}
\beta_{\lambda}(c_{n},c_{n+1},\phi)=\beta_{\lambda}(c_{n},c_{n},\phi)-\delta t_{n}L_{\lambda}(c_{n},\phi) \,\, \text{for all } \, \phi \in \mathcal{V} .
\end{equation}

Let $ c_{n}^{h}$ be the finite element approximation  of $ c_{n}$ (\ref{NEW}).
We obtain the the following FEM adaptive-Newton 
\begin{equation}\label{NEWFEM2}
\beta_{\lambda}(c_{n}^{h},c_{n+1}^{h},\phi)=\beta_{\lambda}(c_{n}^{h},c_{n}^{h},\phi)-\delta t_{n}L_{\lambda}(c_{n}^{h},\phi) \,\, \text{for all } \, \phi \in \mathcal{V}_{h}.
\end{equation}

 By introducing the following notation 
\begin{equation}
c_{n+1}^{(\delta t_{n},h)}:=c_{h}^{n+1}-(1-\delta t_{n})c_{h}^{n}
\end{equation}
and 
\begin{equation}
f^{\delta t_{n}}(c_{n+1}^{h}):=\delta t_{n}f(c_{n}^{h})+f^{'}(c_{n}^{h})(c_{n+1}^{h}-c_{n}^{h} )
\end{equation}
it follows from (\ref{NEWFEM2}) the following result 

\begin{equation}\label{TRANSF}
 \int_{Q}\widetilde{\div}(\widetilde{u}c_{n+1}^{(\delta t_{n},h}))\widetilde{\div}(\widetilde{u}\phi)\mathrm{d}x\mathrm{d}t 
+\lambda\int_{Q}\widetilde{\nabla}c_{n+1}^{(\delta t_{n},h)}\widetilde{\nabla}\phi \mathrm{d}x\mathrm{d}t=\int_{Q}f^{\delta t_{n}}(c_{n+1}^{h})\widetilde{\div}(\widetilde{u}\phi)\mathrm{d}x\mathrm{d}t  \,\, \text{for all } \, \phi \in \mathbb{V}_{h}. 
\end{equation}
We get also the  following result.
\begin{theorem}
\begin{equation}\label{estimation}
\|F_{\lambda}(c_{n+1}^{(\delta t_{n},h)})\|_{\mathbb{V}^{*}} \preceq h^{k}( (\sum \limits_{T \in \mathcal{T}_{h} }\alpha_{T} ^{2})^{\frac{1}{2}}+  (\sum \limits_{T \in \mathcal{T}_{h} }\beta_{T} ^{2})^{\frac{1}{2}}+\lambda(\sum \limits_{T \in \mathcal{T}_{h} }\gamma_{T} ^{2})^{\frac{1}{2}})
\end{equation}
where $$\alpha_{T}=\| ((\widetilde{\div}(\widetilde{u}c_{n+1}^{(\delta t_{n},h}))-f^{\delta t_{n}}(c_{n+1}^{h}))\|_{0,T} \, ,\, \beta_{T}= \|f^{\delta t_{n}}(c_{n+1}^{h})-f(c_{n+1}^{(h,\delta t_{n})}) \|_{0,T} \,\text{and}\,\,  \gamma_{T}=\|\widetilde{\nabla}c_{n+1}^{(\delta t_{n},h)}\|_{0,T}$$
\end{theorem}

\begin{proof}

\begin{equation}\label{stepp1}
\langle F_{\lambda}(c_{n+1}^{(\delta t_{n},h)}),\phi \rangle=\int_{Q}(\widetilde{\div}(\widetilde{u}c_{n+1}^{(\delta t_{n},h})-f(c_{n+1}^{(h,\delta t_{n})}))\widetilde{\div}(\widetilde{u}\phi)\mathrm{d}x\mathrm{d}t +\lambda\int_{Q}\widetilde{\nabla}c_{n+1}^{(\delta t_{n},h)}\widetilde{\nabla}\phi
 \mathrm{d}x\mathrm{d}t.  
\end{equation}
By  adding and subtracting  $ \phi_{h}=\Pi \phi $ in (\ref{stepp1}) and using (\ref{TRANSF}), the following result holds

\begin{eqnarray*}
\langle F_{\lambda}(c_{n+1}^{(\delta t_{n},h)}),\phi \rangle=\int_{Q}(\widetilde{\div}(\widetilde{u}c_{n+1}^{(\delta t_{n},h})-f(c_{n+1}^{(h,\delta t_{n})}))\widetilde{\div}(\widetilde{u}(\phi-\phi_{h}))\mathrm{d}x\mathrm{d}t \\
+\lambda\int_{Q}\widetilde{\nabla}c_{n+1}^{(\delta t_{n},h)}\widetilde{\nabla}(\phi-\phi_{h})\mathrm{d}x\mathrm{d}t +\int_{Q}(f^{\delta t_{n}}(c_{n+1}^{h})-f(c_{n+1}^{(h,\delta t_{n})}))\widetilde{\div}(\widetilde{u}\phi_{h})\mathrm{d}x\mathrm{d}t 
\end{eqnarray*}

\begin{equation}
\int_{Q}((\widetilde{\div}(\widetilde{u}c_{n+1}^{(\delta t_{n},h}))-f^{\delta t_{n}}(c_{n+1}^{h}))\widetilde{\div}(\widetilde{u}(\phi-\phi_{h}))\mathrm{d}x\mathrm{d}t \leq \sum \limits_{T \in \mathcal{T}_{h} }
\| ((\widetilde{\div}(\widetilde{u}c_{n+1}^{(\delta t_{n},h}))-f^{\delta t_{n}}(c_{n+1}^{h}))\|_{0,T}\|\widetilde{\div}(\widetilde{u}(\phi-\phi_{h}))\|_{0,T}
\end{equation}
Applying Cauchy-Schwarz inequality leads to 
\begin{equation}
\int_{Q}((\widetilde{\div}(\widetilde{u}c_{n+1}^{(\delta t_{n},h}))-f^{\delta t_{n}}(c_{n+1}^{h}))\widetilde{\div}(\widetilde{u}(\phi-\phi_{h}))\mathrm{d}x\mathrm{d}t \leq (\sum \limits_{T \in \mathcal{T}_{h} }\|(\widetilde{\div}(\widetilde{u}c_{n+1}^{(\delta t_{n},h})-f^{\delta t_{n}}(c_{n+1}^{h}))\|_{0,T}^{2})^{\frac{1}{2}}(\sum \limits_{T \in \mathcal{T}_{h} } \|\widetilde{\div}(\widetilde{u}(\phi-\phi_{h}))\|_{0,T}^{2})^{\frac{1}{2}}
\end{equation}
and recalling (\ref{divconv}), it follows 

\begin{equation}\label{alpha}
\int_{Q}((\widetilde{\div}(\widetilde{u}c_{n+1}^{(\delta t_{n},h}))-f^{\delta t_{n}}(c_{n+1}^{h}))\widetilde{\div}(\widetilde{u}(\phi-\phi_{h}))\mathrm{d}x\mathrm{d}t \leq C_{u} (\sum \limits_{T \in \mathcal{T}_{h} }\| ((\widetilde{\div}(\widetilde{u}c_{n+1}^{(\delta t_{n},h}))-f^{\delta t_{n}}(c_{n+1}^{h}))\|_{0,T}^{2})^{\frac{1}{2}} h^{k} |\phi|_{k+1,Q}.
\end{equation}
Thus 
\begin{equation}\label{beta}
\int_{Q}(f^{\delta t_{n}}(c_{n+1}^{h})-f(c_{n+1}^{(h,\delta t_{n})}))\widetilde{\div}(\widetilde{u}\phi_{h})\mathrm{d}x\mathrm{d}t \preceq (\sum \limits_{T \in \mathcal{T}_{h} }\alpha_{T}^{2})^{\frac{1}{2}} h^{k} |\phi|_{k+1,Q}.
\end{equation}
It follows also  

$$ \lambda\int_{Q}\widetilde{\nabla}c_{n+1}^{(\delta t_{n},h)}\widetilde{\nabla}(\phi-\phi_{h})\mathrm{d}x\mathrm{d}t \leq  (\sum \limits_{T \in \mathcal{T}_{h} } \|\widetilde{\nabla}c_{n+1}^{(\delta t_{n},h)}\|_{0,T} \| \widetilde{\nabla}(\phi-\phi_{h}) \|_{0,T}.$$
Using inequality (\ref{l2}), the following estimation holds 
\begin{equation}\label{gamma}
\lambda\int_{Q}\widetilde{\nabla}c_{n+1}^{(\delta t_{n},h)}\widetilde{\nabla}(\phi-\phi_{h})\mathrm{d}x\mathrm{d}t \preceq  \lambda h^{k} (\sum \limits_{T \in \mathcal{T}_{h} } \gamma_{T}^{2})^{\frac{1}{2}} |\phi |_{k+1,Q}.
\end{equation}
It follows from (\ref{alpha}),(\ref{beta}) and (\ref{gamma}).
$$\langle F_{\lambda}(c_{n+1}^{(\delta t_{n},h)}),\phi \rangle \preceq h^{k}( (\sum \limits_{T \in \mathcal{T}_{h} }\alpha_{T} ^{2})^{\frac{1}{2}}+  (\sum \limits_{T \in \mathcal{T}_{h} }\beta_{T} ^{2})^{\frac{1}{2}}+\lambda(\sum \limits_{T \in \mathcal{T}_{h} }\gamma_{T} ^{2})^{\frac{1}{2}}) |\phi |_{k+1,Q}.$$
Furthermore $ \|F_{\lambda}(c_{n+1}^{(\delta t_{n},h)})\|_{\mathbb{V}^{*},\mathbb{V}} =\sup \langle F_{\lambda}(c_{n+1}^{(\delta t_{n},h)}),\phi \rangle $ and $ |\phi |_{k+1,Q} \preceq \|\phi\|_{\mathbb{V}},$ then we get the results 
\end{proof}

Since $\delta t_{n}=1$ if the Adaptive-Newton converges   $ \|F_{\lambda}(c_{n+1}^{(\delta t_{n},h)})\|_{\mathbb{V}^{*}}  $ is a reasonnable approximation, moreover under certain conditions on  f we can show that  $ \|c-c_{n+1}^{(\delta t_{n},h)}\|_{\mathbb{V}} $ is equivalent to $ \|F_{\lambda}(c_{n+1}^{(\delta t_{n},h)})\|_{\mathbb{V}^{*}}  $

\subsection{Numerical experiment }
Let we consider  the  following one dimension  hyperbolic conservations laws with linear convection and  stiff sources terms (see \cite{STIFF}).
$$ f(s)=-\mu s (s-1)(s-\frac{1}{2})$$
and initial data 
$$c_{0}(x)= \left \{ \begin{array}{l}
1 \mbox{ if } x \leq 0.3 \\ 0 \ \mbox{if} \, x > 0.3 \  \end{array}
\right .$$ The exact solution approaches the following waves solution $\omega(x-t) $ with

$$ \omega (x)= \left \{ \begin{array}{l}
0 \mbox{ if } c_{0}(x) <  \frac{1}{2} \\ \frac{1}{2}  \ \mbox{if}  \, c_{0}(x) =\frac{1}{2}
\\ 1  \ \mbox{if}  \, c_{0}(x) > \frac{1}{2} \ \end{array}
\right .$$
\begin{exemple}
We first choose $\mu$  such that  $T$ is a contraction for instance $\mu=\frac{1}{7}$, and we will compute the solution of (\ref{eq:stilsn})-(\ref{CD}) by using Picard iteration and simple finite element method  and (\ref{stilpnnl}) by Picard iteration and STILS-MT.   
\begin{figure}[!h]\label{f1}
\includegraphics[scale=0.4]{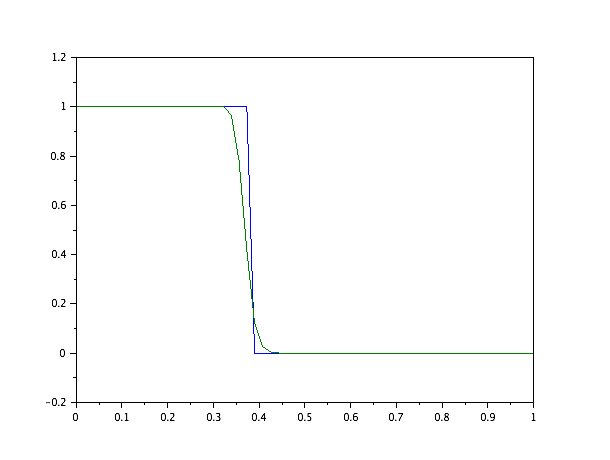}\label{fig1}
\includegraphics[scale=0.4]{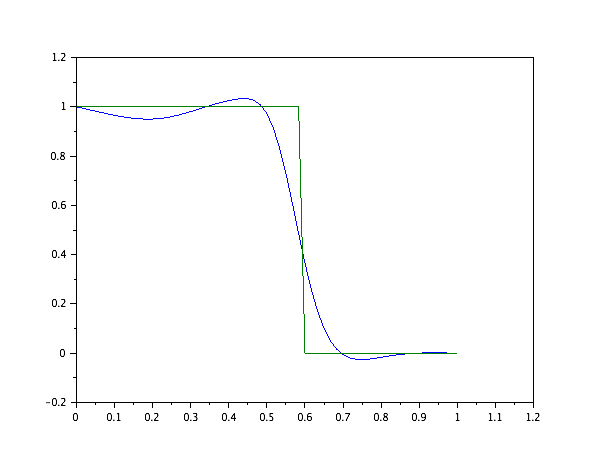}\label{fig2}
\caption{ left Picard's iteration with STILS-MT1 with penalization $\lambda=\frac{5}{12},$ right  Picard's iteration with  simple finite element method    }
\end{figure}
The mesh size of the space  is $\frac{1}{60}$ and the times steep $\frac{1}{65}$ which give $ 60\times 65 $ element in space-time. The solution is presented at $t=\frac{1}{4}$ in figure $1.$  
\end{exemple}
\begin{exemple}
Let we  choose now  $\mu=7$ and  compute the solution of by simple finite element method and 
STILS-MT1 with penalization $\lambda=\frac{5}{12}$ and using Newton Raphson iteration for the semi linearity.
\begin{figure}[!h]\label{f2}
\includegraphics[scale=0.4]{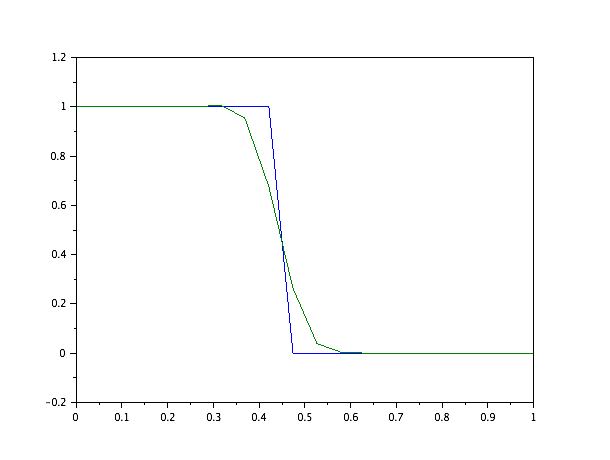}\label{fig3}
\includegraphics[scale=0.4]{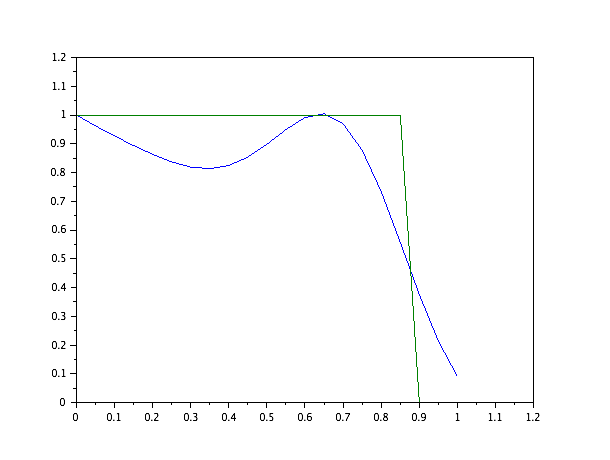}\label{fig4}
\caption{ left Newton Adaptatif Raphson's iteration with STILS-MT1 with penalization $\lambda=\frac{5}{12}$,rigth  Newton adaptaif Raphson's iteration with  simple finite element method    }
\end{figure}
The mesh size of the space  is $\frac{1}{20}$ and the times steep $\frac{1}{25}$ which give $ 20\times 25 $ element in space-time. The solution is presented at $t=\frac{1}{4}.$ 
\end{exemple}
Both  numerical methods can be used   to tame the spurious oscillations produced by STILS-MT and classical finite element methods when advection problem is solved. In the case of  simple finite element methods, we have spurious diffusion  for this  semi-linear conservation, on the other hand the same fact can be obtaining when penalization version is used  but it can be controlled by the parameter $\lambda$. Moreover, STILS-MT can not be used for simple finite element and that gives an important time calculation. We can clearly see that STILS-MT with penalization provides an effective methods   for solving semi linear conservation law numerically.

\end{document}